\documentclass[a4paper]{amsart}

\usepackage{amsmath,amssymb,amsthm,amsfonts}
\usepackage{mathtools} %% add-on of amsmath.
\usepackage[shortlabels]{enumitem}
\usepackage[english]{babel}
\usepackage{microtype}
\usepackage[pdftex]{color}
\usepackage[bookmarks=true,hyperindex,pdftex,colorlinks, citecolor=MidnightBlue,linkcolor=MidnightBlue, urlcolor=MidnightBlue]{hyperref}
\usepackage[dvipsnames]{xcolor}

\usepackage[margin=2.5cm]{geometry}

\newcommand{\Ker}{\operatorname{Ker}}
\newcommand{\Lip}{\operatorname{Lip}}
\newcommand{\Mol}{\operatorname{Mol}}

\newcommand{\norm}[1]{\left\lVert#1\right\rVert}
\newcommand{\conv}{\operatorname{conv}}

\newcommand{\ext}{\operatorname{ext}}

\newcommand{\diam}{\operatorname{diam}}

\newcommand{\FF}{\mathcal F}
\newcommand{\e}{\mathcal \varepsilon}

\newcommand{\Mid}{\operatorname{Mid}}

\newcommand{\PP}{\mathcal P}

\newcommand{\card}{\operatorname{card}}
\newcommand{\RR}{\mathbb{R}}

\newcommand{\EE}{{\mathcal E}}

\newcommand{\<}{\langle}
\renewcommand{\>}{\rangle}

\newtheorem{theorem}{Theorem}[section]
\newtheorem{proposition}[theorem]{Proposition}
\newtheorem{corollary}[theorem]{Corollary}
\newtheorem{lemma}[theorem]{Lemma}

\theoremstyle{definition}
\newtheorem{claim}{Claim}
\newtheorem{definition}[theorem]{Definition}
\newtheorem{example}[theorem]{Example}
\newtheorem{remark}[theorem]{Remark}

\numberwithin{subcase}{case}

\newtheorem{conjecture}[theorem]{Conjecture}
\begin{document}

\title{Geometry and volume product of finite dimensional Lipschitz-free spaces}

\author[M. Alexander]{Matthew Alexander}

\author[M. Fradelizi]{Matthieu Fradelizi}

\author[L. Garc\'ia Lirola]{Luis C. Garc\'ia-Lirola}

\author[A. Zvavitch]{Artem Zvavitch}

\thanks{The first author is supported in part by the Chateaubriand Fellowship of the Office for Science \& Technology of the Embassy of France in the United States and The Centre National de la Recherche Scientifique funding visiting research at Universit\'e Paris-Est Marne-la-Vall\'ee;  the second author supported in part by the Agence Nationale de la Recherche, project GeMeCoD (ANR 2011 BS01 007 01); the  third  author is supported in part by the grants MTM2017-83262-C2-2-P and Fundaci\'on S\'eneca Regi\'on de Murcia 20906/PI/18 also supported by a postdoctoral grant from Fundaci\'on S\'eneca;  the fourth author is supported in part  by the U.S. National Science
Foundation Grant DMS-1101636 and the Bézout Labex, funded by ANR, reference ANR-10-LABX-58.}

\address{Department of Mathematical Sciences, Kent State University,
	Kent, OH 44242, USA} \email{malexan5@kent.edu}\email{lgarcial@kent.edu} \email{zvavitch@math.kent.edu}

\address{Universit\'e Paris-Est Marne-la-Vall\'ee,
Laboratoire d'Analyse et de Math\'{e}matiques Appliqu\'ees (UMR 8050)
Cit\'e Descartes - 5, Bd Descartes, Champs-sur-Marne
77454 Marne-la-Vall\'ee Cedex 2, France}\email{matthieu.fradelizi@u-pem.fr }

\date{April, 2020}

\keywords{Finite metric space, Volume product, Lipschitz-free space, Kantorovich-Rubinstein polytope, Lipschitz polytope}

\subjclass[2010]{Primary: 52A38; Secondary: 52A21, 52A40, 54E45, 05C12}
%	52A21		Finite-dimensional Banach spaces (including special norms, zonoids, etc.) 
% 	52A38  	Length, area, volume
%   52A40  	Inequalities and extremum problems
%	54E45  	Compact (locally compact) metric spaces
%   05C12   Distances in graphs

\begin{abstract}
The goal of this paper is to study geometric and extremal properties of the convex body $B_{\mathcal F(M)}$, which is  the unit ball of the Lipschitz-free Banach space associated with a finite metric space $M$. We investigate $\ell_1$ and $\ell_\infty$-sums, in particular we characterize the metric spaces such that $B_{\mathcal F(M)}$ is a Hanner polytope. We also characterize the finite metric spaces whose Lipschitz-free spaces are isometric. We discuss the extreme properties of the volume product 
 $\mathcal{P}(M)=|B_{\mathcal F(M)}|\cdot|B_{\mathcal F(M)}^\circ|$, when the number of elements of $M$ is fixed. We show that if $\mathcal P(M)$ is maximal among all the metric spaces with the same number of points, then all triangle inequalities in $M$ are strict and $B_{\mathcal F(M)}$ is simplicial. We also focus on the metric spaces minimizing $\mathcal P(M)$, and on Mahler's conjecture for this class of convex bodies. 
 \end{abstract}
\maketitle

\maketitle

\section{Introduction}\label{sec:Intro}
Consider a metric space $(M,d)$ containing a special designated point $a_0$. Such a pair is usually called  a pointed metric space. To this metric space we can associate the space $\Lip_0(M)$ of Lipschitz functions $f:M \to \RR$, with the special property $f(a_0)=0$. We refer to \cite{G, GK, Os, We} for many interesting facts about $\Lip_0(M)$ and its geometry. It turns out that $\Lip_0(M)$ is a Banach space with norm 
\begin{equation*}
\|f\|_{\Lip_0(M)}=\sup\left\{ \frac{f(x)-f(y)}{d(x, y)}, \mbox{  where }  x,y \in M, \mbox{  and } x\not=y\right\}.
\end{equation*}

This space is called the Lipschitz dual of $M$. The closed unit ball of the space $\Lip_0(M)$ is compact for the
topology of pointwise convergence on $M$, and therefore this space has a
canonical predual which is called the Lipschitz-free space over $M$ (also Arens\textendash{}Eells space or transportation cost space) and is denoted by $\FF(M)$. It is the closed subspace of $\Lip_0(M)^*$ generated by the evaluation functionals $\{\delta(x):x\in M\}$ defined by $\delta(x)(f)=f(x)$, for every $f\in \Lip_0(M)$. The case in which $M$ is a finite metric space was studied recently in \cite{DKO, KMO, OO}. 
The goal of this paper is to study the geometry of the unit ball of $\FF(M)$ when $M$ is finite and, in particular, its volume product. 

More precisely, assume our metric space space  $M=\{a_0, \dots, a_{n}\}$, is finite  with metric $d$. Note that each function $f$ in $\Lip_0(M)$ is just a set of $n$ values $f(a_1), \dots, f(a_n)$ and thus we can identify $\Lip_0(M)$ with $\RR^n$, assigning to a function $f \in \Lip_0(M)$ a vector $f=(f(a_1), \dots, f(a_n)) \in \RR^n$. Let us denote  $d_{i,j}=d(a_i, a_j)$, $B_{\Lip_0(M)}$ the unit ball of the $\Lip_0(M)$ and $B_{\mathcal F(M)}=B^\circ_{\Lip_0(M)}$, the unit ball of $\FF(M)$. 
Let $e_1, \dots, e_n$ be the standard basis of $\RR^n$ and let $e_0=0$. For $0\le i\neq j\le n$  let  
\[m_{i,j}=\frac{e_i-e_j}{d_{i, j}}.\]
We denote the set of elementary molecules $m_{i,j}$ 
\[ \Mol(M) = \{m_{i,j}: a_i,a_j\in M, i\neq j\}.\]
Then $\|f\|_{\Lip_0(M)}=\max_{m\in \Mol(M)}\langle f,  m\rangle$. We note that $B_{\mathcal F(M)} = \conv(\Mol(M))$. We shall study the convex bodies in $\RR^n$ that can be obtained as unit ball of a Lipschitz-free space. For example, these bodies have at most $n(n+1)$ vertices, but this is not the only constraint as we shall see. 

In section 2, we explain the correspondence between metric spaces and weighted graphs that is used throughout the article. To each finite metric space $(M,d)$ one can associate a canonical weighted graph $G(M,d)$, the weights on the edges being the distances between the vertices. Then the distance in $M$ corresponds to the shortest path distance in the graph. After removing the unnecessary edges, the edges that appear in the graph model correspond exactly to the vertices of the unit ball of the Lipschitz-free space associated with it. This was recently proved by Aliaga and Guirao \cite{AG} (see also \cite{AP}) in a more general setting: a molecule $m_{i,j}$ is an extreme point of $B_{\mathcal F(M)}$ if and only if $d_{i,j}< d_{i,k}+d_{k,j}$ for every $k\notin \{i,j\}$. We give a simpler proof for finite sets in Section~\ref{sec:graphs}.  We further extend this characterization by proving that the dimension of the face of $B_{\FF(M)}$ which contains $m_{i,j}$ in its interior is exactly the number of points in the geodesics (or the metric segment $[i,j]$) joining $i$ to $j$ in the graph, this is our Theorem~\ref{theor:faces}.  We point out the connection with the so-called Kantorovich-Rubinstein polytopes (studied and popularized in \cite{Ve,GP}). 

In Section \ref{sec:decom}, we explore the conditions for a Lipschitz-free space to be decomposed into $\ell_1$-sum or $\ell_\infty$ sum of spaces. One way of producing an $\ell_1$-sum is to start from two metric spaces $M,N$ and to consider $M\diamond N$ the metric space obtained by identifying the distinguished points of $M$ and $N$. Then one has  $\FF(M\diamond N)=\FF(M)\oplus_1\FF(N)$. In fact we prove in Theorem \ref{th:1decomp} that this is the only possible way to get an $\ell_1$-sum and that such a decomposition is unique. We also characterize $\ell_\infty$-sums in Theorem \ref{theo:inftydecomp} and use these theorems to describe exactly when a Lipschitz free space can be a Hansen-Lima  space, that is a space obtained by taking $\ell_1$ or $\ell_\infty$ sums of $\ell_1^n$ or $\ell_\infty^n$  equivalently when their unit balls are \emph{Hanner polytopes}. 
%Indeed, one may also see the Lipschitz-free mapping as a way to attach to any weighted connected finite graph a finite-dimensional Banach space (the Lipschitz-free space $\FF(M)$ over $M$) or a centrally symmetric convex body ($B_{\mathcal F(M)}=\conv(m_{i,j}: 0\le i\neq j\le n)$). This mapping is no longer an onto map, indeed, not every convex, symmetric body can be associated with a finite metric space. 
 
 In Section \ref{sec:isometries}, we introduce the weighted incidence matrix associated with the canonical graph of a metric space. Its kernel is called the cycle space or first homology group of the graph. Using these notions, we prove that two finite metric spaces have isometric Lipschitz-free spaces if and only if their canonical graphs have the same number of vertices and edges and there is a bijection between their cycles which is constant on the 2-connected components. 

In Section \ref{sec:product}, we investigate the volume product of Lipschitz-free spaces. Recall that the polar body $K^\circ$ of a convex body $K$ is defined by
\begin{equation*}
K^\circ = \{y\in\RR^n : \langle y,x\rangle \le 1 \mbox{\ for all\ } x\in K\}.
\end{equation*}
We note that in most cases in this paper we will assume $K$ to be symmetric.
We write $|A|$ to denote the $k$-dimensional Lebesgue measure (volume)  of a measurable set $A \subset \RR^n$, where $k=1,\dots, n$  is the dimension of the minimal flat containing $A$.
Then the volume product of a symmetric convex body $K$ is defined by 
$$
\PP(K)=|K||K^\circ|,
$$ 
 and the volume product is invariant under invertible linear transformations on $\RR^n$.
%  That is, for any $T \in {\rm GL}(n)$, 
% \begin{equation}\label{eq:affine_invariance}
% \PP(TK)=\PP(K)\quad\text{and}\quad\PP(K^\circ)=\PP(K).
% \end{equation}
The maximum for the volume product in the class of symmetric convex bodies is provided by the Blaschke-Santal\'o inequality: 
$\PP(K) \le \PP(B^n_2)$, for all symmetric convex bodies $K \subset \RR^n$, where $B^n_2$ is the Euclidean unit ball. It was conjectured by Mahler that  the minimum of the volume product is attained at the cube in the class of symmetric convex bodies in $\RR^n$, i.e.\ $\PP(K) \ge \PP(B_\infty^n)=\PP(B_1^n)=\frac{4^n}{n!}$ for all symmetric convex bodies $K \subset \RR^n$ where $B_\infty^n$ is the unit cube and $B_1^n=(B_\infty^n)^\circ$ is a cross-polytope. The case of $n=2$ was confirmed by Mahler himself \cite{Ma1}. Very recently, Iriyeh and  Shibata \cite{IS} gave a positive solution for 
$3$-dimensional symmetric case (see also \cite{FHMRZ}). The conjecture was also proved to be true in several particular special cases, like, e.g., unconditional bodies \cite{SR,Me1,R2}, convex bodies having hyperplane symmetries which fix only one common point \cite{BF},  zonoids \cite{R1,GMR} and bodies of revolution \cite{MR1}. Moreover it was also proved \cite{St, RSW, GM} that bodies with some positive curvature assumption cannot be local minimizers of the volume product. An isomorphic version of the conjecture was proved by Bourgain and Milman \cite{BM}: there is a universal constant $c>0$ such that $\PP(K) \ge c^n\PP(B^n_2)$; see also different proofs in \cite{Ku,Na,GPV}. For more information on Mahler's conjecture and the volume product in general, see expository articles \cite{Sc, RZ}.
In Section \ref{sec:product}, we discuss the maximal and minimal properties of
\begin{equation*}
\mathcal{P}(M):=\mathcal{P}(B_{\mathcal F(M)})
\end{equation*}
where $M$ is a metric space of fixed number of  elements. Our main conjecture here is the following.

\begin{conjecture}\label{conj:minimal} Let $M$ be a metric space with $n+1$ points. Then $\mathcal P(M)\geq \mathcal P(B_1^n)$. 
\end{conjecture}

We will prove a partial result towards Conjecture \ref{conj:minimal}: if $\mathcal P(M)$ is minimal and $B_{\mathcal F(M)}$ is a simplicial polytope, then $M$ is a tree. It was proved  by A. Godard in \cite{Go} that $M$ is a tree if and only if $B_{\mathcal F(M)}$ is an affine image of $B_1^n$, we give a simpler proof of this fact in Section \ref{sec:graphs}. Thus, it would be tempting to say in the conjecture above that the equality is possible if and only if $M$ is a tree.  However, for every $n\geq 3$ we can find a metric space giving equality which is not a tree. The reason is the following. If $M=\{0,1,2,3\}$ is the cycle graph, then $B_{\mathcal F(M)}$ is a linear image of $B_\infty^3$. Now, given a tree $N$ with $n-3$ points then 
$B_{\mathcal F(M\diamond N)}=\conv(B_{\mathcal F(M)}\times\{0\}, \{0\}\times B_{\mathcal F(N)})$ and so we have that 
\[ \mathcal P(M\diamond N) = \frac{3!(n-4)!}{n!}\mathcal P(B_1^{3})\mathcal P(B_1^{n-4}) = %\frac{1}{n}\frac{4^3}{3!}\frac{4^{n-4}}{(n-4)!}=
\frac{4^n}{n!}.\]
For the formula of the volume product of the direct sum of two convex bodies, see e.g. \cite{RZ}. In the previous example, $\mathcal F(M\diamond N)$ is isometric to $\ell_\infty^3\oplus_1 \ell_1^{n-4}$. Hanner polytopes are the conjectured minimizers for the volume product among symmetric convex bodies. 
%In Section \ref{sec:decom}, we characterize the finite metric spaces such that $B_{\mathcal F(M)}$ is a Hanner polytope in terms of the graph associated with $M$. 

It is interesting to note that the maximal value for $\mathcal{P}(M)$ is also an extremely interesting and open problem. Indeed $B_{\mathcal F(M)} \not = B_2^n$  for finite $M$. Thus the maximal case will not follow from the Santal\'o inequality and we must look for other maximum(s) along with the possible conjectured minimizers. For metric spaces of three points, this maximum is attained in the case in which all the distances $d_{i,j}$  between different points coincide. However, this is no longer true for metric spaces with more than three points. This will follow from the result that the maximum of $\mathcal P(M)$ is attained at a metric space such that $B_{\mathcal F(M)}$ is a simplicial polytope. We will also prove that all triangle inequalities in $M$ have to be strict for $\mathcal P(M)$ to be a maximum.

We will begin with a discussion on the connection between graphs and the Lipschitz-free space of the associated metric space, and we will discuss known results concerning these spaces and the relationship between special graphs and metric spaces. In Section~\ref{sec:decom} we will analyze when it is possible to decompose $\mathcal F(M)$ as an $\ell_1$ or an $\ell_\infty$-sum and we will apply that to characterize the metric spaces $M$ such that $B_{\mathcal F(M)}$ is a Hanner polytope. In Section~\ref{sec:isometries}, we characterize the finite metric spaces such that the corresponding Lipschitz-free spaces are isometric. Section~\ref{sec:product} is dedicated to the connections of the structure of metric spaces and the corresponding volume product. In particular, we use the \emph{shadow movement} technique to show in Section~\ref{subsec:max} that the maximum of $\mathcal P(M)$ is attained at a metric space where all triangle inequalities are strict and $B_{\mathcal F(M)}$ is simplicial. In Section~\ref{subsec:complete} we compute the volume product corresponding to the complete graph with equal weights. Finally, in Section~\ref{sec:minimal} we focus on the metric spaces minimizing $\mathcal P(M)$.

\section{Relationship to graphs}\label{sec:graphs}

There is a correspondence between metric spaces and weighted undirected connected graphs. Indeed, to any weighted undirected connected graph $G=(V,E,w)$, with vertices $V$, edges $E$ and weight $w:E\to\RR_+$ one can associate a finite metric space on its set of vertices $V$ by using the shortest path distance. Reciprocally, to any finite metric space $(M,d)$, we can canonically associate a weighted undirected connected finite graph $G(M, d)$ as follows: we first consider the complete weighted graph on $M$ with the weight on the edge between two points being their distance. Then one erases the edge between two points $x,y$ if its weight is equal to the sum of the weights of the edges along a path joining the two points, i.e. if there is $z\neq x,y$ such that $d(x,y)=d(x,z)+d(z,y)$. We will say that $G(M, d)$ is the \emph{canonical graph associated with the metric space $(M,d)$}. We sometimes abuse notation and say that $M$ is a tree, a cycle, etc., meaning that $G(M,d)$ is a tree, a cycle, etc. 

The following lemma describes which weighted graphs are the canonical graph of some metric space, it also yields a bijection between those graphs and finite metric spaces. 

\begin{lemma}\label{lemma:charmetricgraphs} Let $G=(V,E,w)$ be a weighted connected graph. Then there is a metric space such that $G$ is the canonical graph associated with the metric if and only if the following holds:
	for every $\{x,y\}\in E$ and $x_1, \ldots, x_l\in V$ such that $x_1=x, x_l=y$ and $\{x_k, x_{k+1}\}\in E$ for every $k$, we have $0< w(x,y)< \sum_{k=1}^l w(x_k, x_{k+1})$.  
\end{lemma}

\begin{proof}
	Assume first that there is a metric $d$ on $V$ such that $G=G(M,d)$. Then $ 0<d(x,y)=w(x,y)$ for any $\{x,y\}\in E$ and by the triangle inequality $w(x,y)\leq \sum_{k=1}^l w(x_k, x_{k+1})$ for all  $x_1, \ldots x_l$ such that $x_1=x, x_l=y$ and $\{x_k, x_{k+1}\}\in E$ for every $k$. Moreover, assume that  $w(x,y)= \sum_{k=1}^l w(x_k, x_{k+1})$. Then we also have $w(x,y)=w(x,x_2)+w(x_2,y)$ and so $\{x,y\}\notin E$, a contradiction.
	
	Conversely, assume that $G=(V,E)$ satisfies the condition in the statement. Define a metric on $V$ as the shortest path distance:
	\[d(x,y) = \min\left\{\sum_{k=1}^l w(x_k, x_{k+1}) :   x_1=x, x_l=y, \{x_k, x_{k+1}\}\in E \ \forall k\in\{1,\ldots, l\}\right\}\]
	Clearly $d$ is a metric on $V$. We should check that the graph $G'=(V,E')$ associated with $(V,d)$ is $G$. 
	
	Fix $x,y\in V$, $x\neq y$. Then $d(x,y) = \sum_{k=1}^l w(x_k, x_{k+1})$ for some $x_1,\ldots, x_l$ such that  $x_1=x$, $x_l=y$, and $\{x_k, x_{k+1}\}\in E$ for all $k$. Assume that $\{x,y\}\notin E$. Then $l\geq 3$. Moreover, we should have $d(x_2,y) = \sum_{k=2}^l w(x_k, x_{k+1})$ as otherwise we would have a shorter path for the distance between $x$ and $y$. Therefore, $d(x,y)=d(x,x_2)+d(x_2,y)$. So $\{x,y\}\notin  E'$.  
	On the other hand, assume that $\{x,y\}\in E\setminus E'$. Then there is $z\in V\setminus\{x,y\}$ such that $d(x,y)=d(x,z)+d(z,y)$. Take $x_1,\ldots, x_m, x_{m+1}, \ldots, x_l$ such that $x_1=x$, $x_m=z$, $x_l=y$, $d(x,z) = \sum_{k=1}^{m-1} w(x_k, x_{k+1})$, $d(z,y)=\sum_{k=m}^l w(x_k, x_{k+1})$ and $\{x_k, x_{k+1}\}\in E$ for all $k$. Then $d(x,y) = \sum_{k=1}^l w(x_k, x_{k+1})$. Moreover, $d(x,y)\leq w(x,y)$ by the definition of $d$. This is a contradiction with the hypothesis in the statement. Thus $E=E'$, and clearly the weight of the edges in $E'$ is $d(x,y)$. Therefore $G=G(M,d)$. 
	\end{proof}

This representation is very well adapted to our study because the edges that appear in the graph model are exactly corresponding to the vertices of the unit ball of the Lipschitz-free space. This was recently proved by Aliaga and Guirao \cite{AG} for compact metric spaces, and even more recently by Aliaga and Perneck\'a \cite{AP} for complete metric spaces. We give a simpler proof for finite sets below. We denote $\ext K$ the set of extreme points of a convex body $K$, these are precisely the vertices of $K$ when $K$ is a polytope. 

\begin{proposition}[\cite{AG}]\label{extreme}
	Let $(M,d)$ be a pointed finite metric space, with $M=\{a_0,\dots, a_n\}$. Let $G=(M,E,w)$ be the canonical associated with $(M,d)$. %Let $B_{\mathcal F(M)}=\conv(m_{i,j}: 0\le i\neq j\le n)$ be the unit ball of the Lipschitz-free Banach space $\FF(M)$ associated with $(M,d)$. Let $0\le i\neq j\le n$. 
	The following are equivalent.
	\begin{enumerate}[(i)]
		\item $m_{i,j}\notin \ext B_{\mathcal F(M)}$. 
		\item There exists $k\notin\{i,j\}$ such that $d_{i,j}=d_{i,k}+d_{k,j}$.
		\item There exists $k\notin\{i,j\}$ such that $m_{i,j}$ belongs to the segment between $m_{i,k}$ and $m_{k,j}$.
	\end{enumerate}
\end{proposition}

\begin{proof}
	(ii)$\Rightarrow$(iii). Let $k\notin\{i,j\}$ such that $d_{i,j}=d_{i,k}+d_{k,j}$. Then one has
	\[
	m_{i,j}=\frac{e_i-e_j}{d_{i,j}}
	=\frac{d_{ik}}{d_{i,k}+d_{k,j}}\frac{e_i-e_k}{d_{i,k}}+\frac{d_{k,j}}{d_{ik}+d_{kj}}\frac{e_k-e_j}{d_{kj}}
	=\frac{d_{i,k}}{d_{i,k}+d_{k,j}}m_{i,k}+\frac{d_{k,j}}{d_{i,k}+d_{k,j}}m_{k,j}.
	\]
	Hence $m_{i,j}$ belongs to the segment between $m_{i,k}$, and $m_{k,j}$.\\
	(iii)$\Rightarrow$(i). Let $k\notin\{i,j\}$ such that $m_{i,j}$ belongs to the segment between $m_{i,k}$ and $m_{k,j}$. Since $i,j,k$ are distinct, it follows that $m_{i,j}\neq m_{i,k}$ and $m_{i,j}\neq m_{k,j}$ thus $m_{i,j}$ is not an extreme point of $B_{\mathcal F(M)}$.\\
	(i)$\Rightarrow$(ii). If $m_{i,j}$ is not an extreme point of $B_{\mathcal F(M)}$, then $m_{i,j}$ belongs to the relative interior of a face of $B_{\mathcal F(M)}$. Let $\EE=\{(a_k,a_l): 0\le k\neq l\le n\}$. For $e=(a_k,a_l)\in\EE$ we denote $d_e=d_{k,l}$ and $m_e=m_{k,l}$. %By Carath\'eodory's theorem, there exists $\Gamma\subset \EE$, with $2\le \card(\Gamma)\le n$ such that $m_{i,j}\in\conv(m_e:e\in\Gamma)$. 
	Let $\gamma\subset \EE$ be the subset of $\EE$ of smallest cardinality such that $m_{i,j}\in\conv(m_e:e\in\gamma)$. By Carath\'eodory's theorem, $2\le r\le n$, where $r=\card(\gamma)$. Then $S_\gamma:=\conv(m_e:e\in\gamma)$ is an $(r-1)$-dimensional simplex. There exists $(\lambda_e)_{e\in\gamma}$ such that $\lambda_e>0$, for all $e\in\gamma$, $\sum_{e\in\gamma}\lambda_e=1$ and 
	\begin{equation}\label{eq:ext}
	m_{i,j}=\frac{e_i-e_j}{d_{i,j}}=\sum_{e\in\gamma}\lambda_e m_e.
	\end{equation} 
	Note that if $(a_k,a_l)\in \gamma$ then $(a_l,a_k)\notin \gamma$. Indeed, otherwise by triangle inequality we will have 
	\[ 1 = \norm{m_{i,j}} \leq |\lambda_{(a_k, a_l)}-\lambda_{(a_l,a_k)}| + 1-\lambda_{(a_k, a_l)}-\lambda_{(a_l,a_k)}\]
	and so $\lambda_{(a_k,a_l)}=0$, contradicting the minimality of $\gamma$. 
	
	Let $M_\gamma=\{a_k\in M: \exists e\in\gamma; a_k\in e\}$ be the vertices of $M$ which belongs to some of the edges in $\gamma$. Let $\gamma'=\gamma\cup\{(i,j)\}$. We want to prove that the graph $C_\gamma:=(M_\gamma,\gamma')$ is a cycle. First, it is not difficult to see using the minimality of $\gamma$ that the graph $C_\gamma$ is connected. 
	Since $S_\gamma$ is an $(r-1)$-dimensional simplex whose affine hull doesn't contain the origin, its $r$ vertices are linearly independent. Since these vertices are differences of basis vectors, one has at least $r+1$ basis vectors involved in the vertices of $S_\gamma$. Thus, $\card(M_\gamma)\ge r+1$. By linear independence, each $a_k\in M_\gamma$ belongs to at least two edges. By double counting one has 
	$$2(r+1)=2\card(\gamma')=\sum_{a\in M_\gamma}\deg(a)
	\ge 2\card(M_\gamma)\ge 2(r+1).$$
	Hence $\card(M_\gamma)= r+1$ and each vertex of $C_\gamma$ has exactly degree two. Thus $C_\gamma$ is a two-regular connected graph: it is the $(r+1)$-cycle graph. Then it follows from \eqref{eq:ext}  $d_{i,j}=\sum_{e\in\gamma}d_e$. Since $r\ge2$ there exists $k\notin \{i,j\}$ such that $k\in e$ for some $e\in\gamma$. Then one has $d_{i,j}=d_{i,k}+d_{k,j}$.
\end{proof}

One deduces easily the following corollary.

\begin{corollary}\label{coro:extreme}
	Let $(M,d)$ be a pointed finite metric space, with $M=\{a_0,\dots, a_n\}$. Let $G=(M,E,w)$ be the canonical graph associated with $(M,d)$. %Let $B_{\mathcal F(M)}=\conv(v_{i,j}: 0\le i\neq j\le n)$ be the unit ball of the Lipschitz-free Banach space $\FF(M)$ associated with $(M,d)$. 
	Let $0\le i\neq j\le n$. Then
	$m_{i,j}$ is an extreme point of $B_{\mathcal F(M)}$ if and only if $\{a_i,a_j\}\in E$.
	Moreover $B_{\mathcal F(M)}=\conv(m_{i,j}: \{a_i,a_j\}\in E)$ and $B_{\mathcal F(M)}$ has exactly $2\card(E)$ vertices.
\end{corollary}

Indeed, one can go a bit further and relate the number of points in the metric segment (or geodesic)
\[[i,j] = \{k\in \{0, \ldots,n\} : d_{i,j} = d_{i,k}+d_{k,j}\}\]
with the dimension of the face $F$ of $B_{\mathcal F(M)}$ such that $m_{i,j}$ is in the relative interior of $F$. The Proposition~\ref{extreme} says that the dimension of $F$ is $0$ precisely if $[i,j]=\{i,j\}$. In general, we have the following.

\begin{theorem}\label{theor:faces} Let $(M,d)$ be a pointed finite metric space, with $M=\{a_0,\dots, a_n\}$. Given $i\neq j\in\{0,\ldots, n\}$, let $F$ be the face of $B_{\mathcal F(M)}$ such that $m_{i,j}$ belongs to the relative interior of $F$. Then 
	\begin{itemize}
		\item[a)] $F=\conv\{m_{u,v}: d_{i,j}= d_{i,u}+d_{u,v}+d_{v,j}\}$, 
		\item[b)] the dimension of $F$ coincides with the number of points in $[i,j]\setminus\{i,j\}$.  
\end{itemize}
\end{theorem}

As a consequence, $m_{i,j}$ belongs to the interior of a facet of $B_{\mathcal F(M)}$ if and only if $k\in [i,j]$ for all $k\in \{0,\ldots,n\}$. That is a particular case of the study of points of differentiability of the norm of $B_{\mathcal F(M)}$ done in \cite[Theorem~4.3]{PR} for uniformly discrete metric spaces.  

To prove Theorem~\ref{theor:faces}, we will need to check when a set of molecules is linearly independent. The following lemma does the work. 

\begin{lemma}\label{lemma:basis} Let $(M,d)$ be a metric space with $n+1$ elements and $G(M,d)=(V,E,w)$ be its canonical graph. Let $E'\subset E$. Then $\{m_{\gamma}: \gamma \in E'\}$ is a basis of $\mathbb R^n$ if and only if the subgraph with edges $E'$ is a spanning tree of $G(M,d)$.
\end{lemma}

\begin{proof}
	Note that $\operatorname{span}\{m_{\gamma}:\gamma\in E'\}=\operatorname{span}\{\frac{e_{i}-e_j}{d_{i,j}}: \{a_i,a_j\}\in E'\} = \operatorname{span}\{e_i-e_j:\{a_i,a_j\}\in E'\}$. Thus we may assume that $d_{i,j}=1$ for all $\{a_i,a_j\}\in E$.  Let $G'=(V, E')$. Assume first that $G'$ is a spanning tree of $G$. Then $\# E' = n$. Consider the \emph{(vertex-edge) incidence matrix} of $G$, $Q(G')$, which is defined as follows. We consider that every edge is assigned an orientation, which is arbitrary but fixed. The rows and the columns of $Q(G')$ are indexed by $V'$ and $E'$, respectively. The $(i, j )$-entry of $Q(G')$ is $0$ if vertex $i$ and edge $\gamma_j$ are not incident, and otherwise it is $1$ or $-1$ according as $\gamma_j$ originates or terminates at $i$, respectively.   The rank of $Q(G')$ is $n$ since $G'$ is connected (see e.g. Theorem~2.3 in \cite{Ba}). Thus, the $n$ molecules $\{m_e : e \in E'\}$ are linearly independent and so they are a basis of $\mathbb R^n$.
	
	Conversely, assume that $\{m_e : e \in E'\}$ is a basis. Then $\# E'=n$. Moreover, any cycle on $G'$ would provide a linear dependence relation among the molecules $\{m_e: e\in E'\}$. Thus, $G'$ does not contain any cycle. Finally, $G'$ is connected, since otherwise there would be a non-zero Lipschitz function vanishing on $\{m_e:e\in E'\}$. This shows that $G'$ is a spanning tree of $G$.     
\end{proof}

\begin{proof}[Proof of Theorem~\ref{theor:faces}]
	Let $u\neq v$ be such that $d_{i,j}= d_{i,u}+d_{u,v}+d_{v,j}$. It is easy to check that every Lipschitz function attaining its norm on $m_{i,j}$ also attains its norm on $m_{u,v}$. Therefore, $m_{u,v}$ belongs to the intersection of all the exposed faces of $B_{\mathcal F(M)}$ containing $m_{ij}$. That intersection is precisely $F$.
	Conversely, assume that $m_{u,v}\in F$. Consider the $1$-Lipschitz function $f(t):=d(t,a_j)$. Then $\<f, m_{i,j}\>=1$ and so 
	\[m_{u,v}\in F\subset\{\mu\in B_{\mathcal F(M)}: \<f, \mu\> =1\}.\]
	That is, $\<f, m_{u,v}\>=1$. So $d_{u,j}-d_{v,j}=d_{u,v}$. 
	Now, consider the function $g(t):=\frac{d_{i,j}}{2} \frac{d(t,a_j)-d(t,a_i)}{d(t,a_j)+d(t, a_i)}$. This is a 1-Lipschitz function considered first in \cite{IKW2} that turns out to be very useful when studying the extremal structure of free spaces (see e.g. \cite{GPR}), since it only peaks at points in the segment $[i,j]$. As before, we have $\<g, m_{i,j}\>=1$ and so  
	\[m_{u,v}\in F\subset\{\mu\in B_{\mathcal F(M)}: \<g, \mu\> =1\}.\]
	That is, $\<g, m_{u,v}\>=1$. This implies that $d_{i, j}=d_{i, u}+d_{u,j}$. Then
	\[d_{i,j} = d_{i, u}+d_{u,j}= d_{i,u}+d_{u,v}+d_{v,j}.\]
	It follows that
	\[F = \conv\{m_{u,v}: m_{u,v}\in F\} =\conv\{m_{u,v}: d_{i,j}= d_{i,u}+d_{u,v}+d_{v,j}\}. \]
	
	In order to prove b), let $k=\dim F$ and note that $k=\dim\operatorname{span}\{m_{u,v}-m_{i,j}: m_{u,v}\in F\}$. 	 We claim that we only need to consider the vectors $m_{u,v}-m_{i,j}$ where $v\neq j$.  Indeed, it also follows from a) that if $m_{u,v}\in F$ then $m_{i,u}, m_{v,j}\in F$. Note also that
	\[ d_{i, u}(m_{i,u}-m_{x,y})+ d_{u,v}(m_{u,v}-m_{x,y})+d_{v,j}(m_{v,j}-m_{i,j})=0.\]
	Thus, $\dim F=\dim\operatorname{span}\{m_{u,v}-m_{i,j}: m_{u,v}\in F, v\neq j\}$. Since the vector $m_{i,j}$ does not belong to $\operatorname{span}\{m_{u,v}: m_{u,v}\in F, v\neq j\}$, we have that 
	$k= \dim\operatorname{span} A$ 
	where 
	\[A=\{m_{u,v}\in F :  v\neq j\}=\{m_{u,v}:  d_{i,j}= d_{i,u}+d_{u,v}+d_{v,j}, v\neq j\}.\]
	Now, let $G$ be the canonical graph associated with $M$. We claim that $A$ contains $\#([i,j]\setminus\{i,j\})$ linearly independent vectors. Indeed, consider $V' = [i,j]\setminus\{j\}$. For any $v\in V'$, there is a path in $G$ connecting $x$ and $v$, note also that $m_{u,v}\in A$ for any $u$ on that path. Let $G'=(V', E')$ be the subgraph of $G$ obtaining by joining together those paths, and let $G''=(V', E'')$ be a spanning tree of $G'$. Note that $E''$ contains $\#([i,j]\setminus\{i,j\})$ edges. Then  $m_{u,v}\in A$ for any $(u,v)\in E''$ and Lemma \ref{lemma:basis}  tells us that the set $\{m_{u,v}\}_{(u,v)\in E''}$ is linearly independent.  
	
	Finally, given a linearly independent set of molecules $\{m_{u,v} : \{u,v\}\in B\}\subset A$, we have that the graph with edges given by the set $B$ does not contain any cycle. Since its nodes belong to the set $[i,j]\setminus\{j\}$, the cardinality of $B$ is at most that of $[i,j]\setminus\{i,j\}$. This shows that there are at most $\#([i,j]\setminus\{i,j\})$ linearly independent vectors in $A$. 	
\end{proof}

The following nice relationship between trees and affine images of $B_1^n$ was first proved by A. Godard in a more general setting in section~4 of \cite{Go}. We give here a simpler proof in the case of finite metric spaces.

\begin{proposition}[\cite{Go}]\label{prop:tree}
	Let $(M,d)$ be a pointed finite metric space, with $M=\{a_0,\dots, a_n\}$. Let $G=(M,E,w)$ be the canonical weighted undirected connected finite graph associated with $(M,d)$. %Let $B_{\mathcal F(M)}=\conv(v_{i,j}: 0\le i\neq j\le n)$ be the unit ball of the Lipschitz-free Banach space $\FF(M)$ associated with $(M,d)$. 
	Then $G$ is a tree if and only if $B_{\mathcal F(M)}$ is the linear image of $B_1^n$.
\end{proposition}

\begin{proof}
	Recall that the graph $G$ is a tree if and only it is connected and acyclic, equivalently it is connected and $\card(E)=\card(M)-1=n$. Since our graphs are all connected and using Corollary \ref{coro:extreme}, we get that $G$ is a tree if and only if $B_{\mathcal F(M)}$ has $2n$ vertices. But we know that $B_{\mathcal F(M)}$ is full dimensional and is centrally symmetric so it has exactly $2n$ vertices if and only if it is an affine image of $B_1^n$. 
\end{proof}

Given a metric space $M$ and a subset $N$ endowed with the inherited metric so that $a_0\in N$, from $B_{\mathcal F(M)} =\conv\{m_{i,j}: i,j\in M, i\neq j\}$ it is clear that $\mathcal F(N)$ is a subspace of $\mathcal F(M)$ and $B_{\mathcal F(N)}$ is a section of $B_{\mathcal F(M)}$, this will be useful later. The fact that $a_0\in N$ is not important since the Lipschitz-free operation of a metric space $(M,d)$ with two different chosen roots gives two Banach spaces which are isometric can be seen directly in the following way: Assume that $M=\{a_1,\dots, a_{n+1}\}$. Define 
$$K=\conv\left(\left\{\frac{e_i-e_j}{d_{i,j}}: 1\le i\neq j\le n+1\right\}\right)\subset \left\{x\in\RR^{n+1}: \sum_{i=1}^{n+1} x_i=0\right\}.$$
Then $K$ is an $n$-dimensional convex body living in a hyperplane of $\RR^{n+1}$. For $1\le i\le n+1$, denote by $B_{\mathcal F_{a_i}(M)}$ the unit ball  of the Lipschitz-free space $\FF_{a_i}(M)$ pointed at $a_i$. Then $B_{\mathcal F_{a_i}(M)}=P_{e_i^\bot}(K)$ is the orthogonal projection of $K$ on $e_i^\bot$ and this projection is in fact bijective from $\{x\in\RR^{n+1}: \sum_{i=1}^{n+1} x_i=0\}$ onto $\{x\in\RR^{n+1}: x_i=0\}$. Thus for different $i$ and $j$, $B_{\mathcal F_{a_i}(M)}$ and $B_{\mathcal F_{a_j}(M)}$ are affine images of each other. It is also clear from this point that $B_{\mathcal F(N)}$ is (isometric to) a section of $B_{\mathcal F(M)}$ is $N\subset M$. It is interesting to note that there is a lot of literature dedicated to the above $K$, called the \emph{fundamental polytope of the metric space} and also the \emph{Kantorovich-Rubinstein polytope}. It was studied in combinatorics starting from the paper by Vershik \cite{Ve}, see e.g. \cite{GP}. In the particular case of unweighted graphs, it is called the \emph{symmetric edge polytope}, see \cite{DDM} and references therein. 

\section{Decompositions of Lipschitz-free spaces}\label{sec:decom}

\subsection{$\ell_1$-decompositions.} It is interesting to note that  the Lipschitz-free Banach space associated with the series composition of two graphs is the $\ell_1$-sum of their Lipschitz-free Banach spaces, in particular, its unit ball is the convex hull of the two unit balls. 

\begin{definition} Let $(M_i,d_i)$, $i=1,2$, be pointed metric spaces. We denote $M_1\diamond M_2$ the metric space obtained by connecting the graph representations of $M_1$ and $M_2$ by identification of their distinguished points. 
\end{definition}

Note that the definition of $M_1\diamond M_2$ actually depends on the choosen roots. However, we prefer to avoid the cumbersome notation $(M_1,0_1)\diamond (M_2,0_2)$. 

If $M_i$ has $n_i+1$ points, we get that %$\mathcal F(M_1\diamond M_2) = \mathcal F(M_1)\oplus_1\mathcal F(M_2)$ and so 
$B_{\mathcal F(M_1\diamond M_2)}=\conv(B_{\mathcal F(M_1)}\times\{0\}, \{0\}\times B_{\mathcal F(M_2)})\subset \mathbb R^{n_1+n_2}$ and so
\begin{equation}\label{eq:diamond} \mathcal P(M_1\diamond M_2) = \frac{n_1!n_2!}{n!}\mathcal P(M_1)\mathcal P(M_2),
\end{equation} 
this fact will be very useful when studying the extremal values of the volume product. 

Given a finite metric space $M$, we will say that $M$ is \emph{decomposable} if we can find $M_1$ and $M_2$ such that $M=M_1\diamond M_2$. Otherwise we say that $M$ is \emph{indecomposable}. The decomposability of the metric space is closely related to the biconnectedness of its canonical graph. A connected graph $G$ is called \emph{biconnected} if for any vertex $v$ of $G$ the subgraph obtained by removing $v$ is still connected. A \emph{biconnected component} of $G$ is a maximal biconnected subgraph. Any connected graph decomposes into a tree of biconnected components. 

Note that we can write $M=M_1\diamond M_2$ if and only if the canonical graph associated with $M$ is not biconnected. Moreover, in such a case $M_1$ and $M_2$ are the union of biconnected components of $M$. As a consequence, we have the following: 

\begin{proposition}
	Let $M$ be a finite pointed metric space. Then $M$ can be decomposed as $M_1\diamond \ldots \diamond M_r$, where the metric spaces $M_i$ are indecomposable. This decomposition is unique, up to the order of the metric spaces.   
\end{proposition}

It makes sense to wonder if that is the unique way of decomposing a Lipschitz-free space as an $\ell_1$-sum. It turns that this is the case.

\begin{theorem} \label{th:1decomp}
	Let $M=\{a_0,\ldots, a_n\}$ be a finite pointed metric space. Assume that $\mathcal F(M)$ is isometric to $X_1\oplus_1 X_2$. Then there exist metric spaces $M_1, M_2$ (which are obtained from the decomposition of $M$ into its biconnected components) such that $X_k$ is isometric to $\mathcal F(M_k)$.
\end{theorem}

\begin{proof} 
	Let $G=(M, E)$ be the canonical graph associated with the metric space. Let $\phi\colon \mathcal F(M)\to X_1\oplus_1 X_2$ be an isometric isomorphism. Note that 
	\[ \phi(\ext B_{\mathcal F(M)})= \ext B_{X_1}\cup \ext B_{X_2}.\] Thus, if we denote $E_k=\{\{a_i,a_j\}\in E : \phi(m_{i,j})\in \ext B_{X_k}\}$ then $\ext B_{X_k} = \{\phi(m_{i,j}): \{a_i, a_j\}\in E_k\}$, $E=E_1\cup E_2$ and $E_1\cap E_2=\emptyset$.
	
	Now, let $M=\tilde{M}_1\diamond\ldots\diamond \tilde{M}_r$ be the canonical decomposition of $M$ into indecomposable metric spaces, and let $\tilde{E}_i$ be the edges of the $i$-th biconnected component of $M$. We claim that each $\tilde{E}_i$ is either contained in $E_1$ or contained in $E_2$. Note that for every pair of distinct edges in $\tilde{E_i}$ there is a cycle containing them. Thus, it suffices to check that every cycle in $G$ is either contained in $E_1$ or contained in $E_2$. Let $C$ be a cycle in $G$ and assume that $C\cap E_1\neq \emptyset$. We have that
	\[ 0 = \sum_{\{a_i,a_j\}\in C} d_{i,j}m_{i,j} = \sum_{\{a_i,a_j\}\in C\cap E_1} d_{i,j}m_{i,j} + \sum_{\{a_i,a_j\}\in C\cap E_2} d_{i,j}m_{i,j}.\]
	Thus, $\sum_{\{a_i,a_j\}\in C\cap E_1} d_{i,j}m_{i,j}\in X_1\cap X_2 = \{0\}$. This implies that the edges in $C\cap E_1$ form a cycle, and so $C\cap E_2=\emptyset$. This proves the claim. 
	
	For $k=1, 2$, let $I_k = \{i\in \{1,\ldots, r\}: E_i\subset \tilde{E}_k	\}$	and consider $M_k$ the metric space obtained by joining the $\tilde{M}_i$ where $i\in I_k$, for $k=1,2$ with the same identifications considered before. Clearly, $\mathcal F(M)$ is isometric to $\mathcal F(M_1)\oplus_1 \mathcal F(M_2)$. Moreover, the extreme points of $B_{\mathcal F(M_k)}$ correspond, via this isometry, with the molecules $m_{i,j}$ such that $\{i,j\}\in E_k$. Therefore, $\mathcal F(M_k)$ is isometric to $X_k$, and we are done.
\end{proof}

The previous result can be written in an equivalent way that avoids the use of Lipschitz-free spaces: 	

\begin{corollary}
	Let $M=\{a_0,\ldots, a_n\}$ be a finite pointed metric space. Assume that $\Lip_0(M)$ is isometric to $X_1\oplus_\infty X_2$. Then there exist metric spaces $M_1, M_2$ (which are obtained from the decomposition of $M$ into its biconnected components) such that $X_i$ is isometric to $\Lip_0(M_i)$.	
\end{corollary}	

\begin{remark} Given $1<p<\infty$, we have that $\lambda^{1/p} x_1 + (1-\lambda)^{1/p} x_2$ is an extreme point of $B_{X_1\oplus_p X_2}$ for every $\lambda\in [0,1]$ and $x_k\in \ext B_{X_k}$, $k=1,2$. Thus $B_{X_1\oplus_p X_2}$ contains an infinite number of extreme points. Therefore, it is not possible to decompose $\mathcal F(M)$ as $X_1\oplus_p X_2$ for finite $M$. 
\end{remark}

\subsection{$\ell_\infty$-decompositions. } We say that a metric space $M=\{a_0, a_1,\ldots, a_n\}$ is a \emph{spiderweb} if $n=1$ or if the canonical graph associated with $M$ is the complete bipartite graph $K_{2,n-1}$ where all the edges have the same weight. Note that if $M$ has only two points, then it is a (trivial) spiderweb. In addition, the cycle of length $4$ with equal weights is also a spiderweb. The next result shows that for a fixed number of points the spiderweb is the only one metric space whose Lipschitz-free space can be decomposed as an $\ell_\infty$-sum. 

\begin{theorem}\label{theo:inftydecomp} Let $M=\{a_0,\ldots, a_n\}$ be a finite metric space with $n\geq 3$. Then:
	\begin{itemize}
		\item[a)] If $M$ is a spiderweb, then $\mathcal F(M)$ is isometric to $\ell_1^{n-1}\oplus_\infty \mathbb R$. 
		\item[b)] Let $X_1, X_2$ be Banach spaces with $\dim(X_1)\geq \dim(X_2)\geq 1$. Assume that $\mathcal F(M)$ is isometric to $X_1\oplus_\infty X_2$. Then $X_1$ is isometric to $\ell_1^{n-1}$, $X_2=\mathbb R$ and $M$ is a spiderweb.
	\end{itemize}
\end{theorem}

The following lemma is the key to prove Theorem~\ref{theo:inftydecomp}.  It characterizes centrally symmetric faces of the ball of a free space. Given $x,y\in M$, let us denote
\[ \Mid(x,y) = \{z\in M : 2d(x,z)=2d(z,y)=d(x,y)\}. \]

Note that $M$ is a spiderweb precisely if there are $x,y\in M$ such that $\Mid(x,y)=M\setminus\{x,y\}$ and for all $z,u\in M\setminus \{x,y\}$ we have $d(z,u) = d(x,y)$.

\begin{lemma}\label{lemma:symfacet} Let $M$ be a finite pointed metric space. Assume that $F$ is a centrally symmetric face of $B_{\mathcal F(M)}$ with $\dim F\geq 2$. Then one of the following hold:
	\begin{itemize}
		\item[a)] There are distinct points $x,y,u,v\in M$ such that  $d(x,y)=d(u,v)=d(x,v)=d(u,y)$ and \[F=\conv\{m_{x,y}, m_{u,v}, m_{x,v}, m_{u,y}\}.\]
		\item[b)]  There are $x,y\in M$, $x\neq y$, such that 
		\[F = \conv\{m_{x,z}, m_{z,y} : z\in \Mid(x,y)\}.\]
	\end{itemize} 		
\end{lemma}

\begin{proof}
	Let $m_{x,y}\in \ext F\subset \ext B_{\mathcal F(M)}$ and denote $m_{u,v}$ the opposite vertex with respect to the center of $F$. Since $\dim F\geq 2$, there is another vertex $m_{x',y'}$ of $F$, and let $m_{u',v'}$ be its opposite with respect to the center of $F$. Note that the  middle point of the segments $[m_{x,y}, m_{u,v}]$ and $[m_{x',y'}, m_{u',v'}]$ coincide. That is,
	\begin{equation}\label{eq:4mol} m_{x,y} + m_{u,v} = m_{x',y'} + m_{u',v'}\end{equation}
	We distinguish two cases:
	
	\emph{Case 1. $\{x,y\}\cap\{u,v\}=\emptyset$}. Then, up to exchanging the role of $m_{x',y'}$ and $m_{u',v'}$, it follows by linear independence that $x'=x$ and $u'=u$. It follows then that $y'=v$ and $v'=y$, and $d(x,y)=d(u,v)=d(x',y')=d(u',v')$. Since $m_{x',y'}$ was an arbitrary vertex of $F$, we have  $F=\conv\{m_{x,y}, m_{u,v}, m_{x,v}, m_{u,y}\}$. 
	
	\emph{Case 2. $\{x,y\}\cap\{u,v\}\neq\emptyset$}. First assume that $u=x$. It follows from \eqref{eq:4mol} that $x'=u'=x$. That is, 
	\[ m_{x,y} + m_{x,v} = m_{x,y'} + m_{x,v'}.\]
	Then $\{y',v'\}=\{y,v\}$. Since $m_{x,y}\neq m_{x,v}$, we have $y'=v$ and $v'=y$. That is a contradiction. Therefore $u\neq x$. Analogously, $v\neq y$. So either $u=y$ or $v=x$. Let's assume that we are in the case $u=y$. Then $v\neq x$ since $m_{x,y}$ and $m_{y,x}$ cannot belong to the same face, and
	\[ m_{x,y} + m_{y,v} = m_{x',y'} + m_{u',v'}.\]
	Assume first that $d(x,y)\neq d(y,v)$. Then $\operatorname{supp}\{m_{x',y'} + m_{u',v'}\}=\operatorname{supp}\{m_{x,y} + m_{y,v}\}$ has at most three elements and so $d(x',y')=d(u',v')$. But then we also have $d(x,y)=d(y,v)$, a contradiction.
	
	Thus, $d(x,y)=d(y,v)$, and it follows that they are also equal to $d(x',y')$ and $d(u',v')$. Therefore
	\[ e_x-e_v = e_{x'}-e_{y'}+e_{u'}-e_{v'}\]
	and so $\{m_{x',y'}, m_{u',v'}\}=\{m_{x,z}, m_{z,v}\}$ for some $z\in M$ satisfying that $d(x,z)=d(z,v)=d(x,y)=d(y,v)$. That happens for any vertex of $F$. Thus,
	\[ F=\conv(\ext F)\subset \conv\{m_{x,z}, m_{z,v}: d(x,z)=d(z,v)\}\]
	and $m_{x,z}\in F$ if and only if $m_{z,v}\in F$. Let $f\in S_{\Lip_0(M)}$ be an exposing functional for $F$. Pick $z$ such that $m_{x,z}, m_{z,v}\in F$. Then $f(x)-f(z)=d(x,z)$ and $f(z)-f(v)=d(z,v)$. Thus,
	\begin{equation*} d(x,v)\geq f(x)-f(v)=(f(x)-f(z))+(f(z)-f(v))=d(x,z)+d(z,v)\end{equation*}
	and so $z\in \Mid(x,y)$ and $f(x)-f(v)=d(x,v)$. This means that
	\[ F\subset \conv\{m_{x,z}, m_{z,v}: z\in \Mid(x,v)\}.\]
	Moreover, for every $z\in \Mid(x,y)$ we have
	\begin{equation*}
	d(x,v)= f(x)-f(v)=(f(x)-f(z))+(f(z)-f(v))\leq d(x,z)+d(z,v)= d(x,v)\end{equation*}
	and so $m_{x,z}, m_{z,v}\in F$. This shows that 
	\[ F= \conv\{m_{x,z}, m_{z,v}: z\in \Mid(x,v)\}\]
	as desired. 	The case $v=x$ is analogous and we get that  $F=\conv\{m_{u,z}, m_{z,y}: z\in \Mid(u,y)\}$.
\end{proof}

Note that the previous result shows that if $F_1\subset F_2$ are centrally symmetric faces of $B_{\mathcal F(M)}$ and $F_1$ has at least $6$ vertices, then $F_1=F_2$. 

\begin{proof}[Proof of Theorem~\ref{theo:inftydecomp}]
	Assume first that $M$ is a spiderweb, that is, there are $x,y\in M$ such that $\Mid(x,y)=M\setminus\{x,y\}$ and that $m_{u,v}\notin\ext B_{\mathcal F(M)}$ if $u,v\in M\setminus\{x,y\}$. Up to isometry, we may assume that $x$ is the distinguished point of $M$ and $d(x,y)=2$. Consider the $1$-Lipschitz function $f$ given by $f(x)=0$, $f(y)=2$ and $f(z)=1$ otherwise. Let $F=\{m\in B_{\mathcal F(M)}: \<f,m\>=1\}$.
	Note that
	\[F= \conv\{m_{z,x}, m_{y,z}: z\in M\setminus\{x,y\}\}.\]
	Then $\ext B_{\mathcal F(M)}\subset F\cup(-F)$, so $F$ is a facet of $B_{\mathcal F(M)}$. Moreover, 
	\[ \frac{m_{z,x}+m_{y,z}}{2}= m_{y,x}\]
	for all $z\in M\setminus\{x,y\}$, so $F$ is symmetric with respect to $m_{y,x}$. 
	
	Let $K= F-m_{y,x}$. Then $K=B_Y$ for some Banach space $Y$. Therefore,
	\[B_{\mathcal F(M)}=\conv(F\cup(-F)) = B_Y + [m_{x,y}, m_{y,x}]\]
	and so $\mathcal F(M) = Y\oplus_\infty \mathbb R$. Finally, it is clear that $\#\ext(B_Y)=\#\ext(F)=2(n-1)$ and so $Y$ is isometric to $\ell_1^{n-1}$. This proves a).  
	
	Now, assume that $\mathcal F(M)$ is isometric to $X_1\oplus_\infty X_2$ and that $\dim(X_2)\geq 2$. Take $u,v \in \ext B_{X_2}$ such that $[u,v]$ is an edge of $B_{X_2}$. Then both $B_{X_1}+u$ and $B_{X_1}+[u,v]$ are centrally symmetric facets of $B_{X_1\oplus_\infty X_2}$.
	
	\emph{Case 1.} Assume $\dim(X_1)\geq 3$. Then $\#\ext(B_{X_1}+u)=\#\ext(B_{X_1})\geq 6$. Therefore Lemma \ref{lemma:symfacet} yields $B_{X_1}+u=B_{X_1}+[u,v]$, a contradiction.
	
	\emph{Case 2. }	Suppose now that $\dim(X_1)=\dim(X_2)=2$ and so $n=4$. Then $B_{X_1}+[u,v]$ is a facet of $B_{X_1\oplus_{\infty} X_2}$ with at least $8$ vertices. By Lemma \ref{lemma:symfacet}, there are $x,y$ in $M$ such that the set 
	\[  \conv\{m_{x,z}, m_{z,y}: z\in \Mid(x,y)\}\]
	has $8$ vertices. But this is impossible since $\#\Mid(x,y)\leq 3$. 
	
	Therefore, $X_2=\mathbb R$. Then we have $B_{\mathcal F(M)} = B_{X_1} + [-u_0,u_0]$, for a certain vector $u_0$. Thus, $F=B_{X_1}+u_0$ is a centrally symmetric facet of $B_{\mathcal F(M)}$.  Now, we distinguish again two cases:
	
	\emph{Case 1. } Assume $n=3$, that is, $M$ has four points. By Lemma \ref{lemma:symfacet}, we have two possible cases:
	
	\emph{Case 1.1.} $F=\conv A$, where $A=\{m_{x,y}, m_{u,v}, m_{x,v}, m_{u,y}\}\subset \ext B_{\mathcal F(M)}$, and $d(x,y)=d(u,v)=d(x,v)=d(u,y)$. Then  $M=\{x,y,u,v\}$ and $\#\ext B_{X_1}=\#\ext F=4$. Therefore, $B_{\mathcal F(M)}$ has precisely $8$ vertices, the ones in $A\cup(-A)$. That is, the edges of the canonical graph associated with $M$ are $\{x,y\}$, $\{u,v\}$, $\{x,v\}$ and $\{u,y\}$, and all of them have the same weight. So $M$ is a cycle of length $4$ with equal weights, in particular, a spiderweb.
	
	\emph{Case 1.2.} There are $x,y\in M$ such that $F=\conv\{m_{x,z}, m_{z,y}: z\in \Mid(x,y)\}$. Since $F$ has at least $4$ vertices, we have that $\Mid(x,y)=M\setminus\{x,y\}$ and $\#\ext F=4$. Then $\#\ext B_{\mathcal F(M)}=8$. It follows that $M$ is again a cycle of length $4$ with equal weights.  
	
	\emph{Case 2. } Assume $n\geq 4$. We have $\#\ext F=\#\ext B_{X_1}\geq 2(n-1)\geq 6$. Then the first case in Lemma \ref{lemma:symfacet} does not hold. Therefore, 
	there are $x,y\in M$ such that $F = \conv A$, where $A=\{m_{x,z}, m_{z,y} : z\in \Mid(x,y)\}$.  Note that $\ext F\subset A$ and $\ext B_{\mathcal F(M)}\subset \ext F\cup \ext(-F)$.  It follows that for every $z\in M\setminus\{x,y\}$ we have $\{x,z\}, \{z,y\}\in E$, and all the edges are of this form. Thus, $M$ is a spiderweb. Moreover, it follows that $\#\ext B_{X_1}=2(n-1)$ and so $X_1$ is isometric to $\ell_1^{n-1}$.

\end{proof}

Combining Proposition~\ref{prop:zonotopes} and Theorem~\ref{theo:inftydecomp}, we get the following result. 

\begin{corollary}\label{cor:ellinfty}
	Let $M=\{a_0,\ldots, a_n\}$ be a finite pointed metric space. Then $\mathcal F(M)$ is isometric to $\ell_\infty^n$ if and only if $n\leq 2$ and $M$ is tree, or $n=3$ and $M$ is a cycle with equal weights. 
\end{corollary}

\subsection{Zonotopes.} Godard \cite{Go} characterized the metric spaces $M$ such that $\mathcal F(M)$ is isometric to a subspace of $L_1$ as those metric spaces which embed into an $\mathbb{R}$-tree. In the finite-dimensional setting, the embeddability into $L_1$ is equivalent to the fact that the dual ball $B_{\Lip_0(M)}$ is a zonoid \cite{Bo} (see also \cite{Sc,Ko, RZ}). Let us recall that a convex body is said to be a \emph{zonotope} if it is a finite Minkowski sum of segments, and it is said to be a \emph{zonoid} if it is the limit, in the Hausdorff distance, of a sequence of zonotopes. Both notions coincide for polytopes. Zonoids enjoy very nice geometrical properties, for instance, all faces are also zonoids and thus centrally symmetric. We refer the reader to \cite{Sc, GW} for more results about zonoids and zonotopes. 

Zonoids satisfy Mahler's conjecture \cite{R1,GMR}. Thus, it makes sense to wonder about when $B_{\mathcal F(M)}$ is a zonoid. Unfortunately, there are few cases when that happens, as the following result shows. 

\begin{proposition}\label{prop:zonotopes} Let $M=\{a_0, \ldots, a_n\}$ be a finite pointed metric space. The following are equivalent:
	\begin{enumerate}[i)]
		\item $B_{\mathcal F(M)}$ is a zonotope.
		\item $n\leq 2$, or $n=3$ and $M$ is a cycle graph (of length $4$) with equal weights.
	\end{enumerate} 
\end{proposition}

\begin{proof}
	ii)$\Rightarrow$ i). Every $2$-dimensional symmetric polytope is a zonotope. In addition, if $n=3$ and the canonical graph associated with $M$ is the cycle graph with equal weights then $B_{\mathcal F(M)}$ is a linear image of the cube $B_\infty^3$.
	
	i)$\Rightarrow$ ii). Let $F$ be a (centrally symmetric) facet of $B_{\mathcal F(M)}$. Assume that $n\geq 4$.  Since $F$ is a zonoid of dimension $n-1$, it has at least $2^{n-1}$ vertices. In addition, by Lemma \ref{lemma:symfacet}, there are $x,y$ such that $F=\conv\{m_{x,z},m_{z,y}: d(x,z)=d(z,y)\}$. Then 
	\[2^{n-1}\leq \#\ext F\leq 2\#\{z\in M : d(x,z)=d(z,y)\}\leq 2(n-1),\]
	a contradiction. Thus, $n\leq 3$. In the case $n=2$ there is nothing to prove. Thus, we may assume that $n=3$. 
	Let $G=(V,E,w)$ be the canonical graph associated with $M$.  
	We distinguish some cases:
	
	\emph{Case 1. $G$ has at least one leaf, say $a_0$}. Then $\mathcal F(M)$ is a $\ell_1$-sum of $\mathbb R$ and another space. It is easy to check that then all the facets of $B_{\mathcal F(M)}$ are not centrally symmetric. 
	
	\emph{Case 2. There is a node in $G$ with degree $3$.} We may assume that the node is $a_0$. Consider the $1$-Lipschitz function $f$ given by $f(a_i)=d_{i,0}$. Note that if $i,j\neq 0$  are distinct then $|f(a_i)-f(a_j)|<d_{i,j}$. Thus,  $\<f, m_{i,j}\>=1$ if, and only if, $m_{i,j}=m_{i,0}$ and $i\neq 0$. That is, the face of $B_{\mathcal F(M)}$ given by $f$ has dimension $2$ and an odd number of vertices, so it is not centrally symmetric, a contradiction.

	Since the previous cases lead to a contradiction, we have that $G$ is a cycle, so $B_{\mathcal F(M)}$ has $8$ vertices. Thus, $F$ has $4$ vertices. It follows that, for any $\{x,y\}\in E$, either $m_{x,y}\in F$ or $m_{y,x}\in F$. Lemma \ref{lemma:symfacet} says that, in any case, the set $\{d(x,y): \{x,y\}\in E\}$ is a singleton. That is, all the edges have the same weight, as desired. 
\end{proof}

\subsection{Hanner polytopes.} A symmetric convex body $K$ is called a \emph{Hanner polytope} if $K$ is one-dimensional, or it is the $\ell_1$ or $\ell_\infty$ sum of two (lower dimensional) Hanner polytopes. They are the unit balls of the Hansen-Lima spaces \cite{HL}. Hanner polytopes are the conjectured minimizers for the volume product, moreover they are known to be local minimizers \cite{NPRZ, Ki}. We finish the section characterizing for which metric spaces the ball of the Lipschitz-free space is a Hanner polytope. 

\begin{theorem}\label{th:hanner} Let $M$ be a finite pointed metric space. The following are equivalent:
	\begin{itemize}
		\item[i)] $B_{\mathcal F(M)}$ is a Hanner polytope.
		\item[ii)] $M=M_1\diamond \ldots\diamond M_r$, where each of the $M_i$ is a spiderweb.
	\end{itemize}	
\end{theorem} 

\begin{proof} ii)$\Rightarrow$i) follows from the fact that $B_{\mathcal F(M_i)}$ is a Hanner polytope provided $M_i$ is a spiderweb and the connection between the volume product of $B_{\mathcal F(M)}$ and the one of the $B_{\mathcal F(M_i)}$ given in \eqref{eq:diamond}.
	
	i)$\Rightarrow$ii). We prove the result by induction on $n=\dim \mathcal F(M)$. For $n\leq 2$, every Hanner polytope is a linear image of $B_1^n$. So $M$ is a tree, which is a sum of spiderwebs consisting on two points. Now, assume $n\geq 3$ and that the result is true for metric spaces with at most $n$ points. Write $\mathcal F(M) = X_1\oplus_p X_2$, where $p\in \{1,\infty\}$ and $B_{X_1}$, $B_{X_2}$ are Hanner polytopes.  If $p=\infty$, then $M$ is a spiderweb by Theorem~\ref{theo:inftydecomp}, and we are done. Otherwise, $p=1$. Let $M=M_1\diamond\ldots \diamond M_r$ be the decomposition of $M$ into its biconnected components. Then by Theorem~\ref{th:1decomp}, there are metric spaces $N_1 = M_{i_1}\diamond \ldots \diamond M_{i_s}$ and $N_2=M_{j_1}\diamond\ldots\diamond M_{j_{r-s}}$, with $\{1,\ldots, r\}=\{i_1,\ldots, i_s, j_1, \ldots, j_{r-s}\}$, such that $X_i$ is isometric to $\mathcal F(N_i)$ for $i=1,2$. The induction hypothesis says that both $N_1$ and $N_2$ are the sum of spiderwebs. Since the decomposition of $N_1$ and $N_2$ into their biconnected components is unique, it follows that each of the $M_i$ is a spiderweb, as desired.
\end{proof}

\section{Isometries of Lipschitz-free spaces}\label{sec:isometries}

It is known that if the canonical graph associated with a metric space $M$ is a complete graph, and $\mathcal F(M)$ is isometric to $\mathcal F(M')$, then $M$ is isometric to $M'$ \cite[Theorem~3.55]{We}, equivalently, $B_{\mathcal F(M)}$ is a linear image of $B_{\mathcal F(M')}$. This does not hold in general, for instance $\mathcal F(M)=\ell_1^n$ whenever $M$ is a tree. Indeed, the Lipschitz-free spaces over two metric spaces $M$ and $M'$ are isometric whenever $M$ and $M'$ have the same decomposition into indecomposable metric spaces. Our next goal is to characterize when two finite metric spaces $M, M'$ have isometric Lipschitz-free spaces.

First, given a metric space with canonical graph $G=(V,E,d)$, we identify $\mathcal F(M)$ with a quotient of $\ell_1(E)$. To this end, consider 
the operator $\partial : \RR^{E}\to\RR^{V}$ defined by $\partial \chi_e=m_e$ for any $e\in E$ and extended to $\RR^{E}$ by linearity, where $\chi$ denotes the indicator function, so for $u,v\in V$, $\chi_u:\RR^V\to \RR$ and $\chi_u(u)=1$ and $\chi_u(v)=0$ for $v\neq u$. The \emph{cycle space} is the kernel of $\partial$. This kernel is also called the first homology group $H_1(G)=H_1(G,\RR)$.

\begin{theorem}\label{th:quotient} Let $M$ be a finite metric space and $G=(V,E,d)$ be its canonical graph. Then $\mathcal F(M)$ is isometric to $\ell_1(E)/\Ker\partial$.
\end{theorem}
\begin{proof} 
	We start by fixing an orientation of the edges. This orientation means that each edge $e\in E$ has a vertex $e_+$ which is a target and a vertex $e_-$ which is a source and one has $e=(e_-,e_+)$ and moreover the molecule associated with $e$ is well defined, it is $m_e=\frac{\chi_{e_+}-\chi_{e_-}}{d(e)}$. Then 
	\[ B_{\mathcal F(M)}=\conv\{\pm m_e: e\in E\}\subset \RR^V\cap\{\varphi\in\RR^V: \sum_{v\in V} \varphi(v)=0\}:=H\]
	where $H$ is a hyperplane in $\RR^V$. Moreover $B_{\mathcal F(M)}$ is a convex set of dimension $|V|-1$, so of full dimension in this hyperplane and it has exactly $|E|$ vertices. 
	Let us now see how one can compute the norm in $\mathcal F(M)$ in relation with the $\ell_1$ norm in $\RR^{E}$ defined by $\|f\|_1=\sum_{e\in E}|f(e)|$. For $\varphi\in \RR^V$, since $B_{\mathcal F(M)}=\conv\{\pm m_e; e\in E\}$, one has 
	\begin{eqnarray*}
		\|\varphi\|_{\mathcal F(M)}&=&\inf\{t>0; \varphi\in tB_{\mathcal F(M)}\}=\inf\{t>0; \varphi=t\sum_{e\in E} t_em_e; \sum_{e\in E} |t_e|=1\}\\
		&=&\inf\{\sum_{e\in E} |t_e|; \varphi=\sum_{e\in E} t_em_e\}=\inf\{\sum_{e\in E} |t_e|; \varphi=\partial(\sum_{e\in E} t_e\chi_e)\}\\
		&=&\inf\{\|f\|_1; \varphi=\partial f\}.
	\end{eqnarray*}
	This implies that for every $f\in\RR^{E}$ one has 
	$$\|\partial f\|_{\mathcal F(M)}=\inf\{\|g\|_1; \partial f=\partial g\}=\inf\{\|f+h\|_1; h\in\Ker\partial\}=\|f\|_{\ell_1(E)/\Ker\partial}.$$
	Hence one has $\mathcal F(M)=\ell_1(E)/\Ker\partial$. Another way of seeing this is to say that $B_{\mathcal F(M)}$ is the projection of $B_1^{E}$, the $\ell_1$ ball in $\RR^{E}$. 
\end{proof}

Note that the previous result also says that $\Lip_0(M)$ is isometric to the subspace $(\Ker \partial)^\bot$ of $\ell_\infty (E)$. Indeed, the isometry is given by the map
\[ f\mapsto \left(\frac{f(e^+)-f(e^-)}{d(e)}\right)_{e\in E}\]

Now, we can characterize which finite metric spaces have the same Lipschitz-free space. Given two graphs $G$ and $G'$, we say that a bijection between their sets of edges is \emph{cyclic} if it induces a bijection between the cycles of $G$ and the cycles of $G'$. 

\begin{theorem}\label{th:isometries} Let $M$, $M'$ be finite metric spaces with canonical graphs $G=(V, E, d)$ and $G'=(V', E', d')$. The following are equivalent:
	\begin{enumerate}[i)]
		\item $\mathcal F(M)$ is isometric to $\mathcal F(M')$, that is, there is a linear isomorphism $T$ such that $TB_{\mathcal F(M)}=B_{\mathcal F(M')}$. 
		\item $|V|=|V'|$, $|E|=|E'|$, and $H_1(G)$ and $H_1(G')$ are isometric to the same subspace of $\ell_1^m$, where $m=|E|=|E'|$.
		\item $|V|=|V'|$ and there is a cyclic bijection $\sigma\colon E\to E'$ such that the function $e\mapsto d(\sigma(e))/d(e)$ is constant on each $2$-connected component of $G$.  
	\end{enumerate}
	Moreover, any isometry $T\colon \mathcal F(M)\to \mathcal F(M')$ is induced by a cyclic bijection $\sigma\colon E\to E'$, meaning that $T(m_e) = \pm m'_{\sigma(e)}$ for all $e\in E$.
\end{theorem}

Notice that, in particular, we get that two unweighted graphs have the same Lipschitz-free space if and only if there is a cyclic bijection between their edges. In such a case, the graphs are said to be \emph{cyclically equivalent} or \emph{$2$-isomorphic}. H. Whitney \cite{Wh} showed that two graphs are cyclically equivalent if and only if one can pass from one to the other by the successive application of two operations which are:
\begin{itemize}
	\item[(a)] Vertex gluing and vice versa. Vertex gluing of $M$ and $M'$ is the same as what we denoted by $M\diamond M'$. The opposite operation, the separation of a graph into two non-connected components, can be done only at a ``separating vertex''. 
	\item[(b)] Switching: starting from a pair $(u,v)$ of separating vertices one exchanges $u$ and $v$.
\end{itemize}

\begin{proof}[Proof of Theorem \ref{th:isometries}]
	ii)$\Rightarrow$ i). By Theorem \ref{th:quotient},  $\mathcal F(M)=\ell_1(E)/\Ker H_1(G)$ and $\mathcal F(M')=\ell_1(E')/\Ker H_1(G')$, these spaces are isometric.
	
	i)$\Rightarrow$ii) and iii). If $B_{\mathcal F(M')}=T B_{\mathcal F(M)}$ then $\conv\{\pm m'_{e’}; e'\in E’\}=\conv\{\pm Tm_e; e\in E\}$, where $m_e=\frac{\chi_{e_+}-\chi_{e_-}}{d(e)}$, for $e\in E$ and $m’_{e’}=\frac{\chi_{e'_+}-\chi_{e'_-}}{d’(e')}$, for $e’\in E’$. Hence these convex sets have the same dimension and the same number of vertices. This implies that $|V|=|V’|$, $|E|=|E’|=m$. Moreover since $T$ establishes a bijection among the extreme points, there exists a unique bijection $\sigma:E\to E'$ and a map $\e: E\to\{-1,1\}$ such that for any $e\in E$ one has $T(m_e)=\e(e)m'_{\sigma(e)}$. We claim that $\sigma$ is a cyclic bijection. Indeed, let $\{e_1,\ldots, e_k\}$ be a cycle in $E$. Then $\sum_{i=1}^k d(e_i)m_{e_i} = 0$ and so \begin{equation}\label{eq:isomcycles}0=\sum_{i=1}^k d(e_i) \epsilon(e_i) m_{\sigma(e_i)}=\sum_{i=1}^n \frac{d(e_i)}{d(\sigma(e_i))}\epsilon(e_i)(\chi_{e_i^+}-\chi_{e_i^-}).\end{equation}
	Then the degree of any vertex in the subgraph of $G'$ induced by $\{\sigma(e_1), \ldots, \sigma(e_k)\}$ is odd and so  $\{\sigma(e_1), \ldots, \sigma(e_k)\}$ is union of cycles in $G'$. If $I\subset\{1,\ldots, k\}$ is so that $\{\sigma(e_i):i\in I\}$ is a cycle in $G'$, then applying the same argument with $T^{-1}$ we get that $\{e_i: i\in I\}$ is union of cycles in $G$. Thus $I=\{1,\ldots, k\}$ and the image of the cycle $\{e_1,\ldots, e_k\}$  is a cycle in $G'$. From \eqref{eq:isomcycles} we have that $d(e_i)/d(\sigma(e_i))=d(e_j)/d(\sigma(e_j))$ for all $i, j$. Thus, the function $e\mapsto d(\sigma(e))/d(e)$ is constant on each cycle, and so on each $2$-connected component of $G$.

	Now we prove ii). We denote $\partial: \RR^{E}\to\RR^V$ defined by $\partial \chi_e=m_e$ and $\partial': \RR^{E'}\to\RR^{V'}$ defined by $\partial' \chi_{e'}=m'_{e'}$. From Theorem \ref{th:quotient} one has $\mathcal F(M)=\ell_1(E)/\Ker\partial$ and $\mathcal F(M')=\ell_1(E')/\Ker\partial'$. Let us define $\tilde{T}:\RR^{E}\to\RR^{E'}$ by $\tilde{T}(\chi_e)=\e(e)\chi_{\sigma(e)}$ and extended by linearity. Then for any $e\in E$ one has 
	\[ T(\partial(\chi_e))=T(m_e)=\e(e)m'_{\sigma(e)}=\e(e)\partial'(\chi_{\sigma(e)})=\partial'(\e(e)\chi_{\sigma(e)})=\partial'(\tilde{T}(\chi_e)).\]
	Thus $T\circ\partial=\partial'\circ \tilde{T}.$
	And since $T$ and $\tilde{T}$ are bijective this implies that 
	\[ H_1(G)=\Ker(\partial)=\Ker(T\circ \partial)=\Ker(\partial'\circ \tilde{T})=(\tilde{T})^{-1}(\Ker(\partial '))=(\tilde{T})^{-1}(H_1(G')).\]
	Thus $H_1(G')=\tilde{T}(H_1(G))$, where $\tilde{T}$ is an isometry from $\ell_1(E)$ to $\ell_1(E')$.
	
	iii)$\Rightarrow$i). Clearly $\sigma$ preserves the $2$-connected components of the graph. Since $\mathcal F(M)$ is isometric to $\mathcal F(M_1)\oplus_1\ldots\oplus_1\mathcal F(M_r)$ provided that $M=M_1\diamond \ldots\diamond M_r$, we may assume that $G$ is $2$-connected and so there is $a>0$ such that $d(\sigma(e))= ad(e)$ for all $e\in E$. We claim that there is  an orientation on the edges of $G'$ such that ``$\sigma$ preserves directed cycles'', i.e. for every $e_1,\ldots e_k\in E$ and $\varepsilon_1,\ldots, \varepsilon_k\in \{-1,1\}$ such that $\sum_{i=1}^k \varepsilon_i (\chi_{e_i^+}-\chi_{e_i^-})=0$, we also have $\sum_{i=1}^k \varepsilon_i (\chi_{\sigma(e_i)^+}-\chi_{\sigma(e_i)^-})=0$.
	
	The 2-connectness of $G$ and Whitney's theorem imply that, up to a graph automorphism, every cyclic bijection is a composition of switchings. Thus, it suffices to prove the case in which $\sigma$ is a switching on a pair $(u,v)$ of separating vertices of $G$, that is, we can write $E= E_1\cup E_2$ and $E'=E_1\cup E_2'$, where $\{u,x\}\in E_2\iff \{v,x\}\in E_2'$ and $\{v,x\}\in E_2\iff \{u,x\}\in E_2'$. Define $\sigma(e)^+:= u$ if $e^+ =v$ and $e\in E_2$, $\sigma(e)^+:= v$ if $e^+ =u$ and $e\in E_2$, and  $\sigma(e)^+:= e^+$  otherwise, this orientation does the work. 
	
	 Let $\tilde{E}$ be a subset of $E$ giving a spanning tree.  Since $\sigma$ preserves cycles, $\sigma(\tilde{E})$ is a spanning tree for $G'$, and we can choose an orientation of $G'$ in the same way. 	Now, we can define a linear map $T\colon \mathcal F(M)\to \mathcal F(M')$ given by $T(m_e)=m_{\sigma(e)}$ if $e\in\tilde{E}$ and extended by linearity. Since $\tilde{E}$ and $\sigma(\tilde{E})$ are spanning trees of $G$ and $G'$, $\{m_e:e\in \tilde{E}\}$ is a basis on $\mathcal F(M)$ and $\{m_{\sigma(e)}:e\in \tilde{E}\}$ is a basis of $\mathcal F(M')$. Thus, $T$ is an isomorphism. It remains to check that every vertex of $B_{\mathcal F(M)}$ maps to a vertex of $B_{\mathcal F(M')}$. This is clearly true for the ones of the form $\pm m_e, e\in \tilde{E}$. Let $e\in E\setminus \tilde{E}$. Then there are $e_1, \ldots e_r\in \tilde{E}$ such that $\{e_1, \ldots, e_r, e\}$ is a cycle in $G$ and $\chi_{e^+}-\chi_{e^-} =\sum_{i=1}^r \varepsilon_i(\chi_{e_i^+}-\chi_{e_i^-})$. Then, thanks to the choice of the orientation of the edges in $G'$, we have $\chi_{\sigma(e)^+}-\chi_{\sigma(e)^-} =\sum_{i=1}^r \varepsilon_i(\chi_{\sigma(e_i)^+}-\chi_{\sigma(e_i)^-})$ and so
	
	\begin{align*}T(d(e)m_e) &= T(\chi_{e^+}-\chi_{e^-})= \sum_{i=1}^r \varepsilon_i T(\chi_{e_i^+}-\chi_{e_i^-}) = \sum_{i=1}^r\varepsilon_i d(e_i)T(m_{e_i}) =\sum_{i=1}^r \varepsilon_i d(e_i)m_{\sigma(e_i)} \\
	&=\sum_{i=1}^r \varepsilon_i \frac{d(e_i)}{d(\sigma(e_i))} (\chi_{\sigma(e_i)^+}-\chi_{\sigma(e_i)^-}) 
	= \frac{1}{a}(\chi_{\sigma(e)^+}-\chi_{\sigma(e)^-})= d(e)m_{\sigma(e)}
	\end{align*}
	and so $T(m_e)=m_{\sigma(e)}$, which is a vertex of $B_{\mathcal F(M')}$ as desired. 
\end{proof}

	In the case of 3-connected graphs, the result of Whitney mentioned before yields every cycle bijection $\sigma\colon E\to E'$ is induced by a vertex isomorphism, that is, there is $f\colon V\to V'$ such that $\sigma(e)=(f(e^+), f(e^-))$ for all $e$. As a consequence, we obtain the following result. 

\begin{corollary} Let $(M,d)$ and $(M',d)$ be metric spaces such that $TB_{\mathcal F(M)}=B_{\mathcal F(M')}$ for a linear isomorphism $T$. Assume that $G(M,d)$ is $3$-connected. Then $M$ and $M'$ are isometric, moreover, there is an isometry $f\colon M\to M'$ inducing $T$, i.e. $T(m_e)=m'_{f(e^+), f(e^-)}$ for all any edge $e$ in $G(M,d)$.
\end{corollary}

Note that the previous corollary improves the result of Weaver \cite[Theorem 3.55]{We} in the case of finite metric spaces.

\section{Extremal properties of the volume product}\label{sec:product}

We will focus now on the maximal and the minimal value of $\mathcal P(M)=|B_{\mathcal F(M)}|\cdot|B_{\Lip_0(M)}|$. It is a well-known fact that the volume product of convex bodies is invariant under linear isometries. We start with the easy observation that $\mathcal P(M)$ is invariant under dilations on $M$. Given two metric spaces $(M,d)$ and $(N,\rho)$, we say that a map $f\colon M\to N$ is a \emph{dilation} if it is a bijection and  there is a constant $a>0$ such that $\rho(f(x), f(y))=a d(x,y)$ for every $x,y\in M$.

\begin{proposition}\label{prop:invdil}
	$\mathcal P(M)$ is invariant under dilations of $M$. 
\end{proposition}

\begin{proof}
	Let $(M,d)$ and $(N, \rho)$ be finite metric spaces and denote $\delta_M\colon M\to \mathcal F(M)$ and $\delta_N\colon N\to\mathcal F(N)$ the canonical embeddings. Assume that $f\colon M\to N$ is a bijection such that $\rho(f(x),f(y))=ad(x,y)$ for some $a>0$ and every $x,y\in M$. By the fundamental property of Lipschitz-free spaces (see e.g. \cite{We}), there is a bounded linear operator $\hat{f}\colon \mathcal F(M)\to \mathcal F(N)$ such that $\hat{f}\circ \delta_M = \delta_N\circ f$. Therefore
	\begin{align*} \hat{f}(B_{\mathcal F(M)})&= \hat{f}\left(\conv\left\{\frac{\delta_M(x)-\delta_M(y)}{d(x,y)}:x,y\in M, x\neq y\right\}\right) \\
	 &=\conv\left\{\frac{\hat{f}(\delta_M(x))-\hat{f}(\delta_M(y))}{d(x,y)}: x,y\in M, x\neq y\right\} \\
	 &= a\conv\left\{\frac{\delta_N(f(x))-\delta_N(f(y))}{\rho(f(x),f(y))}: x,y\in M, x\neq y\right\} = aB_{\mathcal F(N)}
	\end{align*}
	That is, $B_{\mathcal F(N)}$ is a linear image of $B_{\mathcal F(M)}$. Then $\mathcal P(N)=\mathcal P(M)$. 	 
\end{proof}

%Given a finite metric space $M$, one can think in the points of $M$ as the vertices of a weighted graph, where the weight of an edge is the distance between the points. Then one removes all the edges between points $x,y$ such that there is another point $z$ so that $d(x,z)+d(z,y)=d(x,y)$. Then the distance in $M$ is the same as the shortest path distance in the graph. 

The main idea behind  a number of proofs of our results in this section is the shadow system technique: a {\it shadow system of convex sets} along a direction $\theta \in S^{n-1}$ is a family of convex sets $L_t \in \RR^n$ which are defined by
$$
L_t=\conv\{x+ \alpha(x) t \theta :  x \in B\},
$$
where $B\subset \RR^n$ is a bounded set, called the basis of the shadow system, $\alpha:B\to \RR$ is a bounded function, called the speed of the shadow system, and $t$ belongs to an open interval in $\RR$. We say that a shadow system is {\it non-degenerate}, if all the convex sets $L_t$ have non-empty interior. Shadow systems were first introduced by  Rogers and Shephard \cite{RS}. Campi and Gronchi \cite{CG} proved that if $L_t$ is a symmetric shadow system then $t\mapsto |L_t^\circ|^{-1}$ is a convex function of $t$.  In \cite{MR2},  Meyer and Reisner  generalized  this result to the non-symmetric case and studied the equality case. The following proposition summarize those results in symmetric case:

\begin{proposition}[\cite{CG, MR2}]\label{lm:mr}
Let $L_t$, $t\in [-a,a]$, be a non-degenerate shadow system in $\RR^n$, with direction $\theta\in {\mathbb S}^{n-1}$,
 then $t\mapsto \left|L_t^\circ\right|^{-1}$ is a convex function on $[-a,a]$. \\
If, moreover, $t\mapsto |L_t|$ is affine on $[-a,a]$ and  $t\mapsto \PP(L_t)$ is constant on $[-a,a]$, then there exists $w\in\RR^n$ and $\alpha\in\RR$, such that for every $t\in[-a,a]$, one has $L_t=A_t(L_0)$, where $A_t:\RR^n\to\RR^n$ is the affine map defined by
$$A_t(x)=x+t(w\cdot x+\alpha)\theta.$$
\end{proposition}

\subsection{Maximal Case}\label{subsec:max}

We will show that the maximum of the volume product is attained at a metric space such that its canonical graph is a weighted complete graph. In dimension two we can say something more.

\begin{theorem} For all metric spaces of three elements, $\PP(M)\leq \PP(K_3)$, where $K_3$ is the metric space corresponding to a complete graph with equal weights.\end{theorem}\label{th:max2}
\begin{proof}
The result follows immediately from Theorem~4 in \cite{AFZ}, but here we present a self-contained proof. Let $M=\{0,1,2\}$ be a metric space of three elements. By Proposition~\ref{prop:invdil} we may assume that $d_{1,2}=1$, so $B_{\mathcal F(M)}=\conv\{\pm \frac{e_1}{d_{1,0}}, \pm \frac{e_2}{d_{2,0}}, \pm (e_1-e_2)\}$. Then after a linear transformation we get the body $L=\conv\{\pm e_1, \pm e_2, \pm (d_{1,0}e_1-d_{2,0}e_2)\}$ with the same volume product. Note that $|d_{1,0}-d_{2,0}|\leq 1$. 
Consider the shadow system
\[L_t = \conv\left\{\pm e_1, \pm e_2, \pm \left(\frac{d_{1,0}+d_{2,0}}{2}(e_1-e_2)+t(e_1+e_2)\right)\right\}\]
Then $t\mapsto |L_t|$ is constant on $[-1/2,1/2]$ and $|L_t^\circ|=|L_{-t}^\circ|$. Thus, $\mathcal P(L_t)\leq \mathcal P(L_0)$ for all $t\in [-1/2,1/2]$. In particular, $\mathcal P(M)=\mathcal P(L)\leq \mathcal P(L_0)$. Now, direct computation shows that $|L_0| = 1+2r$ and $|L^\circ_0|=\frac{4r-1}{r^2}$, where $r= \frac{d_{1,0}+d_{2,0}}{2}$. By simple calculus, $\PP(L_0)\leq \PP(K_3)=9$. 
\end{proof}

To deal with the maximum in general dimension we need two results about shadow systems. The following lemma is an extension of Lemma 3.2 of \cite{AFZ}, where a similar result is proved in the case in which $F$ is a facet of $K$. 

\begin{lemma}\label{lemma:movingpoint} Let $K\subset \mathbb R^n$ be a convex body with $0\in \operatorname{int} K$ and $F$ be a face of $K$ of dimension $0\leq k\leq n-1$. Let $\{F_i\}_{i=1}^{m}$ be the facets of $K$ containing $F$, with unit normal vectors $\{u_i\}_{i=1}^{m}$. Let $x_F\in \operatorname{relint} F$ and $v\in \mathbb S^{n-1}$ be such that $\<v,u_i\>>0$ for all $i=1,\ldots, m$. Take $x_t=x_F+tv$ and $K_t=\conv(K, x_t)$. Then, if $t>0$ is small enough,
	\[ |K_t|=|K|+\frac{t}{n}\sum_{i=1}^{m} |F_i|\<v,u_i\>\] 
and $|K_t^\circ| = |K| + o(t^{k+1})$.
\end{lemma}

\begin{proof} Let $\{F_i\}_{i=m+1}^l$ be the facets of $K$ that do not contain $F$, with normal unit vectors $\{u_i\}_{i=m+1}^l$. Let us denote by $h_K(u)=\sup\{\<x,u\> :x\in K\}$ the support function of $K$. Note that $\<x_F, u_i\><h_K(u_i)$ for every $i=m+1,\ldots, l$ since $x_F\in \operatorname{relint} F$. Thus, there is $\varepsilon>0$ such that we have $\<x_t, u_i\><h_K(u_i)$ for every $i=m+1,\ldots, l$  if $t\in[0,\varepsilon]$. In addition, note that
\[ K_t = \conv(\ext K \cup\{x_t\})\] 
and so $\ext K_t \subset \ext K\cup\{x_t\}$. 

\begin{claim} $K_t= K\cup \bigcup_{i=1}^{m} \conv(F_i, x_t)$ if $t\in[0,\varepsilon]$.\end{claim}

It is clear that $K\cup \bigcup_{i=1}^{m} \conv(F_i, x_t)\subset K_t$. Conversely, it suffices to check that every edge of $K_t$ is contained in $K\cup \bigcup_{i=1}^{m} \conv(F_i, x_t)$. Let $p,q$ be such that the segment $[p,q]$ is an edge of $K_t$. Then $p,q\in\ext K_t$. If both of them belong to $K$, then $[p,q]\subset K$. Otherwise, we may assume that $p=x_t$ and $q\in \ext K$. Assume that $q\notin \bigcup_{i=1}^{m} F_i$. Then $\<q, u_i\> < h_K(u_i)$ for all $i\in 1,\ldots, m$. Thus,
\[ \<\lambda x_t +(1-\lambda) q, u_i\> < h_K(u_i) \quad \forall i\in\{1,\ldots, m\}\] 
when $\lambda>0$ is small enough. Moreover, 
\[ \<\lambda x_t +(1-\lambda) q, u_i\> < h_K(u_i) \quad \forall i\in\{m+1,\ldots,l\}\]
for all $0<\lambda<1$. % since $\<x_t, u_i\><h_K(u_i)$.
This means that $[p,q]\cap \operatorname{int} K\neq \emptyset$, so $[p,q]$ also intersects $\operatorname{int}K_t$ and so it is not an edge. Thus, $q\in \bigcup_{i=1}^{m} F_i$. Therefore $[p,q]\subset \conv(F_i,x_t)$ for some $i\in\{1,\ldots,m\}$ as desired. This proves the claim. 

\begin{claim}\label{cl_a} $\conv(F_i,x_t)\cap \conv(F_j, x_t) = \conv(F_i\cap F_j,x_t)$ if $i,j\in \{1,\ldots,m\}$.\end{claim}

Let $y\in \conv(F_i,x_t)\cap \conv(F_j, x_t)$. We can write
\[ y = \lambda x_t+(1-\lambda)y_i = \mu x_t+(1-\mu)y_j \]
where $y_i\in F_i$, $y_j\in F_j$, and $0\leq \lambda, \mu\leq 1$. Clearly we may assume $\lambda,\mu<1$. Assume first that $\mu>\lambda$. If $y_i \in  F_j$, then $y\in\conv(F_i\cap F_j, x_t)$ and we are done. So we may assume $\<y_i, u_j\> < h_K(u_j)$. Then we have
 \[ \<y, u_j\> = \lambda \<x_t, u_j\>+(1-\lambda)\<y_i, u_j\> < \lambda \<x_t, u_j\> + (1-\lambda)h_K(u_j). \]
 On the other hand,
 \[\<y, u_j\> =  \mu\<x_t, u_j\> + (1-\mu) \<y_j, u_j\>  =\mu\<x_t, u_j\> + (1-\mu)h_K(u_j) \]
 so $(\mu-\lambda) \<x_t, u_j\> < (\mu-\lambda)h_K(u_j)$, which yields $\<x_t, u_j\> <h_K(u_j)$, a contradiction. If $\mu<\lambda$, we reach to a similar contradiction  multiplying by $u_i$. Finally, assume that $\mu = \lambda<1$. Then $y_i=y_j\in F_i\cap F_j$ and we are done. This proves the claim. 
 
  It follows from  Claim \ref{cl_a} that $\conv(F_i, x_t)\cap \conv(F_j, x_t)$ has empty interior if $i\neq j$. Thus,
  \[|K_t| = |K| + \sum_{i=1}^{m} |\conv(F_i, x_t)| =|K| + \frac{1}{n}\sum_{i=1}^{m} |F_i|(\<x_t, u_i\>-h_K(u_i)) = |K| + \frac{t}{n}\sum_{i=1}^{m} |F_i|\<v, u_i\>. \] 
  Now, we focus on the volume of the polar body. Note that the vertices of $K^\circ$ are the points $\{u_i/h_K(u_i)\}_{i=1}^l$. Let $F'=\{x\in K^\circ: \<x,y\>= 1 \ \forall y\in F\}$ be the face of $K^\circ$ corresponding to $F$. Then $F'$ is the convex hull of the vertices of $K^\circ$ which belong to $F'$ and so $F'=\conv(\{u_i/h_K(u_i)\}_{i=1}^{m})$.

For each $i\in\{1,\ldots, m\}$, let $\{v_s^i\}_{s=1}^{r_i}$ be the vertices of $K^\circ$ which are adjacent to $u_i/h_K(u_i)$ and which do not belong to $\{u_j/h_K(u_j)\}_{j=1}^l$. Let \[
\eta=\min\{1-\<v_s^i, x_F\> : i\in \{1, \ldots, m\}, s\in\{1, \ldots, r_i\}\}>0.\]  
  Note that $u_i/h_K(u_i)\in K_t^\circ$ if and only if $\<u_i/h_K(u_i), x_t\> \leq 1$ if and only if  $i\in\{m+1,\ldots,l\}$. This means that 
  \[K^\circ = K_t^\circ \cup \conv(u_i/h_K(u_i), w_s^i : i\in\{1,\ldots, m\}, s\in \{1,\ldots r_i\}),\]
   where $w_s^i$ is the unique point in the segment $[u_i/h_K(u_i), v_s^i]$ with $\<w_s^i, x_t\>=1$. Write 
   \[w_s^i = (1-\lambda_s^i) \frac{u_i}{h_K(u_i)} + \lambda_s^i v_s^i\]
   where $0\leq\lambda_s^i\leq 1$. An easy computation shows that
	\[\lambda_s^i = \frac{t\<\frac{u_i}{h_K(u_i)},v\>}{\<\frac{u_i}{h_K(u_i)}-v_s^i, x_t\>} \leq \frac{t\<\frac{u_i}{h_K(u_i)},v\>}{1-\<v_s^i, x_F\> -\varepsilon \left|\frac{u_i}{h_K(u_i)}-v_s^i\right|} \leq  \frac{t\<u_i,v\>}{h_K(u_i)(\eta-\varepsilon\diam(K^\circ))}\leq ct \]
  	for some constant $c>0$, provided $\varepsilon$ is small enough. 
  	Therefore, 
  	\[|u_i/h_K(u_i) - w_s^i| = \lambda_s^i |u_i/h_K(u_i) - v_s^i| \leq c\diam(K^\circ)t =c't\]
  	for each $i=1,\ldots, m$ and $s=1,\ldots, r_i$. This means that 
  	\[ K^\circ\setminus K_t^\circ \subset \conv\left(\bigcup_{i=1}^{m} B(u_i/h_K(u_i), c't)\right)=\conv(\{u_i/h_K(u_i)\}_{i=1}^{m})+c'tB = F'+c'tB.\] 
  	Now, by Steiner's formula (see e.g.~\cite{Sc}), 
  	\[|F'+c'tB|=\sum_{j=0}^n (c't)^{n-j}\kappa_{n-j}V_i(F') = o(t^{k+1})\] 
  	since the affine dimension of $F'$ is $n-k+1$ (see e.g.~\cite[pag.~50]{Gr}). Thus, $|K^\circ\setminus K_t^\circ| = o(t^{k+1})$.  
\end{proof}

\begin{lemma}\label{lemma:movingpointsym} Let $K\subset \mathbb R^n$ be a symmetric convex body and $F$ be a face of $K$ of dimension $1\leq k\leq n-1$. Let $\{F_i\}_{i=1}^{m}$ be the facets of $K$ containing $F$, with unit normal vectors $\{u_i\}_{i=1}^{m}$. Let $x_F\in \operatorname{relint} F$ and $v\in \mathbb S^{n-1}$ be such that $\<v,u_i\>>0$ for all $i=1,\ldots, m$. Take $x_t=x_F+tv$ and $K_t=\conv(K, x_t, -x_t)$. Then, $\mathcal P(K_t)>\mathcal P(K)$ if $t>0$ is small enough.
\end{lemma}

\begin{proof} Let $L_t=\conv(K, x_t)$. It is easy to check that for each $i,j\in\{1,\ldots, m\}$, the set $\conv(F_i,x_t)\cap \conv(-F_j, -x_j)$ has empty interior. Now, the argument in the proof of the previous lemma gives that 
	\[(L_t \cap (-L_t)) \setminus K= \bigcup_{i,j=1}^{m} \conv(F_i,x_t)\cap \conv(-F_j, -x_t)\]
and so $|(L_t\cap (-L_t))\setminus K|=0$. By Lemma \ref{lemma:movingpoint} we have
		\[ |K_t|=|K|+|L_t\setminus K|+|(-L_t)\setminus K|=|K|+2\frac{t}{n}\sum_{i=1}^{m} |F_i|\<v,u_i\>.\] 
On the other hand, $K_t^\circ = \{x\in K : |\<x,x_t\>|\leq 1\}= L_t^\circ\cap (-L_t)^\circ$ and so
\[ |K^\circ\setminus K_t^\circ| \leq  |K^\circ\setminus L_t^\circ| + |K^\circ\setminus (-L_t)^\circ| = o(t^{k+1})\]
 if $t>0$ is small enough. Thus $|K^\circ_t|=|K^\circ|+o(t^{k+1})$. 
Therefore, 
\[\mathcal P(K_t) = \mathcal P(K) + 2|K^\circ|\frac{t}{n}\sum_{i=1}^{m} |F_i|\<v,u_i\> + |K|o(t^{k+1})> \mathcal P(K)\]
when $t$ approaches $0$. 
\end{proof}

We remark that the previous lemma does not work when $k=0$. Indeed, if $K$ is a square in $\mathbb R^2$ and $x_F$ is a vertex of $K$ then we can have $\mathcal P(K_t)=\mathcal P(K)$ for all $t>0$. 

Now we can show that the maximum of $\mathcal P(M)$ is attained at a weighted complete graph. 

\begin{theorem}\label{th:alltriaglesstrict} Let $M=\{a_0,\ldots, a_n\}$ be a metric space such that $\mathcal P(M)$ is maximal among the metric spaces with the same number of elements. Then $d_{i,j}<d_{i,k}+d_{k,j}$ for every distinct points $a_i,a_j,a_k$.
\end{theorem}

\begin{proof} Assume the contrary. Consider the set 
	\[ A = \{(i,j)\in \{0,\ldots,n\}^2 : i\neq j,  d_{i,j}=d_{i,k}+d_{k,j} \text{ for some } k\in \{0,\ldots,n\}\setminus\{i,j\}\}.\]
Take $(i,j)\in A$ such that $d_{i,j}$ is maximal. Consider $d^t$ given by $d^t_{i,j}=\frac{d_{i,j}}{1+t}$ and $d^t_{u,v}=d_{u,v}$ otherwise. The maximality of $(i,j)$ implies that $d_{i,j}>\max\{|d_{i,k}-d_{k,j}|: k\in \{0,\ldots,n\}\setminus\{i,j\}\}$. Thus, $d^t$ is a metric on $M$ for $t$ small enough. Note that $\frac{e_i-e_j}{d^t_{i,j}}=(1+t)m_{i,j}$. Thus,
\[ B_{\mathcal F(M,d^t)} = \conv(\{m_{u,v}: \{u,v\}\neq \{i,j\}\}, \pm(1+t)m_{i,j}) = \conv(B_{\mathcal F(M)}, \pm (1+t)m_{i,j}).\]

Note that $m_{i,j}$ is not an extreme point of $B_{\mathcal F(M)}$ since $(i,j)\in A$. Thus, $m_{i,j}\in \operatorname{relint} F$ for some face $F$ of $B_{\mathcal F(M)}$ with dimension $k$, $1\leq k\leq n-1$. Let $F_1,\ldots F_{m}$ be the facets of $B_{\mathcal F(M)}$ containing $F$, with normal vectors $\{u_s\}_{s=1}^{m}$. Note that $\<m_{i,j}, u_s\>=h_{B_{\mathcal F(M)}}(u_s)>0$. 	
Therefore, Lemma \ref{lemma:movingpointsym} with $x_F= v=m_{i,j}$ yields that $\mathcal P(M, d^t)>\mathcal P(M,d)$ if $t>0$ is small enough, a contradiction.	
\end{proof}

It is natural to wonder if the maximum of the volume product is attained at the unweighted complete graph. We showed in Theorem~\ref{th:max2} that this is the case for metric spaces with three points. However, that is no longer the case in higher dimensions. In order to prove that, we will show that the maximum of $\mathcal P(M)$ is attained at a metric space such that $B_{\mathcal F(M)}$ is simplicial.

The following can be proved by using the same arguments as in Theorem~3.4 in \cite{AFZ}.

\begin{proposition}\label{prop:adap34}
	Let $\mathcal K$ be a family of centered convex polytopes and let $K\in\mathcal K$ be such that $\mathcal P(K)$ is maximal among the elements of $\mathcal K$. Assume that $\conv(K, (1+t)x, -(1+t)x)\in \mathcal K$ for every vertex $x$ of $K$ and $t>0$ small enough. Then $K$ is simplicial. 
\end{proposition}

%\begin{proof}
%	Let $x$ be a vertex of $K$. Denote $F_x=\{y\in K^\circ : \<y,x\>=1\}$ the corresponding facet of $K^\circ$ and $\mathcal F(x)$ the set of faces of $K$ containing $x$. Define $K_t= \conv(K, (1+t)x, -(1+t)x)$. As in \cite[Theorem 3.4]{AFZ}, for small values of $t>0$ we have that
%	\[ |K_t|=|K|+2t\sum_{F\in\mathcal F(x)} |\conv(F,0)|,  \]
%	\[|K_t^\circ| =|K^\circ|-2nt|\conv(F_x,0)|+o(t)\]
%	and so
%	\[ \mathcal P(K_t) = \mathcal P(K) + 2t\left(|K^\circ|\sum_{F\in\mathcal F(x)} |\conv(F,0)|-n|K||\conv(F_x,0)|\right)+o(t)\]
%	Since $K_t\in \mathcal K$ for $t$ small enough, we have $\mathcal P(K_t)\leq P(K)$ and so 
%	\[|K^\circ|\sum_{F\in\mathcal F(x)} |\conv(F,0)|\leq n|K||\conv(F_x,0)|.\]
%	Now, the same argument as in \cite[Theorem 3.4]{AFZ} gives that $K$ is simplicial and that 
%	\[|K^\circ|\sum_{F\in\mathcal F(x)} |\conv(F,0)|= n|K||\conv(F_x,0)|\]	
%	for any vertex $x$ of $K$. 
%\end{proof}

In the particular case of free spaces, we have the following. 

\begin{theorem}\label{prop:maxsimplicial}
	Let $M=\{a_0,\ldots, a_n\}$ be a finite pointed metric space such that $\mathcal P(M)$ is maximal among the metric spaces with the same number of elements. Then $B_{\mathcal F(M)}$ is simplicial.  
\end{theorem}

\begin{proof}
Note that every extreme point of $B_{\mathcal F(M)}$ is of the form $m_{i,j}$ for some $a_i,a_j\in M$, $i\neq j$. By Theorem~\ref{th:alltriaglesstrict}, we have $d_{i,j}>\max\{|d_{i,k}-d_{k,j}|: k\in \{0,\ldots,n\}\setminus\{i,j\}\}$. Thus, $d^t$ given by $d^t_{u,v}= d_{u,v}$ if $\{u,v\}\neq\{i,j\}$ and $d^t_{i,j} = d^t_{j,i} = \frac{d_{i,j}}{1+t}$ defines a metric on $M$ for $t$ small enough. Moreover, 
\[ B_{\mathcal F(M,d^t)} = \conv(\{m_{u,v}: \{u,v\}\neq \{i,j\}\}, \pm(1+t)m_{i,j}) = \conv(B_{\mathcal F(M)}, \pm (1+t)m_{i,j}).\]
This shows that the hypotheses of Proposition~\ref{prop:adap34} hold and so $B_{\mathcal F(M)}$ is simplicial.
\end{proof}

\subsection{The special case of complete graph with equal weights}\label{subsec:complete}

Let $K_{n+1}$ denote the metric space of $n+1$ elements such that $d_{i,j}=1$ if $i\neq j$, i.e. the metric space associated with the complete graph where all the weights are equal to $1$. 

Note that $B_{\mathcal F(K_{n+1})}$ is not simplicial whenever $n\geq 3$. Indeed, consider the $1$-Lipschitz function given by $f(a_i)=1$ if $i\in \{2,\ldots, n\}$ and $f(a_0)=f(a_1)=0$. %Clearly, $f\in S_{\Lip_0(K_{n+1})}$.
Note that the $2n-2$ molecules $m_{i,0}$, $m_{i,1}$, $i\in \{2,\ldots,n\}\}$ are vertices of $B_{\mathcal F(K_{n+1})}$ that belong to the face of $B_{\mathcal F(K_{n+1})}$ exposed by $f$. Since $2n-2>n$, this face is not a simplex.  Therefore, it follows from Proposition~\ref{prop:maxsimplicial} that $K_{n+1}$ is not a metric space with maximum volume product if $n\geq 3$.

We also would like to provide a computation for the volume product of $K_{n+1}$.  Note first that $B_{\mathcal F(K_{n+1})}=\conv\{\pm e_i, \pm (e_i-e_j); 1\leq i\neq j\leq n\}$ has exactly $n(n+1)$ vertices.

Let us describe more precisely the unit ball %$B^\circ_{\mathcal F(K_{n+1})}$
of $\Lip_0(K_{n+1})$.

\begin{claim} 
$$B_{\Lip_0(K_{n+1})}=\conv\left\{\pm \sum_{i\in I} e_i: I\subset\{1,\dots,n\}\right\}.$$ 
\end{claim}

\begin{proof} Denote by $C$ the right hand side set. First let us prove that $C\subset B_{\Lip_0(K_{n+1})}$. One has 
$$
B_{\Lip_0(K_{n+1})}=B^\circ_{\mathcal F(K_{n+1})}=\{x\in \RR^n: |x_i|\leq 1, |x_i-x_j|\le1\ \text{ for all } 1\le i\neq j\le n\}.
$$ 
For any $I\subset\{1,\dots,n\}$ let us denote $x(I)=\sum_{i\in I} e_i$. For any $k,l\in\{1,\dots,n\}$ one has $x(I)_k\in\{0,1\}$ and $x(I)_k-x(I)_l\in\{-1, 0,1\}$ hence $x(I)\in B^\circ_{\mathcal F(K_{n+1})}$. Therefore $C\subset B_{\Lip_0(K_{n+1})}$.\\  To show that $B^\circ_{\mathcal F(K_{n+1})}\subset C$ we consider $x\in B_{\Lip_0(K_{n+1})}$. Then $|x_i|\leq 1$ and so we may assume, reordering our axes if necessary, that $-1\leq x_1 \leq x_2 \leq \ldots \leq x_{n-1}\leq x_n\leq 1$. Further, since $|x_i-x_j|\leq 1$ we get $x_n-x_1\leq 1$.  Now let us consider the indices with positive entries and negative entries separately.  That is, we let $k$ be the last negative index, choosing it to be 0 if no entries are negative, and to be $n$ if all entries are negative.  Then 
\begin{align*}
x=\left(x_2-x_1\right)(-e_1)+\left(x_3-x_2\right)\left(-e_1-e_2\right)+\ldots +\left(x_k-x_{k-1}\right)\left(-\sum_{i\leq k-1}e_i\right) +(-x_k)\left(-\sum_{i\leq k} e_i\right)
\\+x_{k+1}\sum_{i\geq k+1} e_i+ \left(x_{k+2}-x_{k+1}\right)\sum_{i\geq k+2} e_i+\ldots + \left(x_n-x_{n-1}\right)e_n,
\end{align*}
thus $\frac{1}{x_n-x_1}x$ is a convex combination of points in $C$. Therefore $B_{\Lip_0(K_{n+1})}\subset C$.
\end{proof}

\begin{claim}\label{cl_zono} $$B_{\Lip_0(K_{n+1})}=\frac{1}{2} B_\infty^n+\frac{1}{2} [-\sum_{i=1}^n e_i,\sum_{i=1}^n e_i].$$
\end{claim}

\begin{proof} Denote by $D$ the zonotope on the right hand side. First let us take an extreme point $x\in D$ then 
$$x=\frac{1}{2} \sum_{i=1}^n  \varepsilon_i  e_i + \frac{1}{2} \varepsilon_{n+1} \sum_{i=1}^n e_i=\sum_{i=1}^n\frac{\varepsilon_i+\varepsilon_{n+1}}{2}e_i.$$ where the $\varepsilon_i\in\{-1,1\}$. Let $I=\{i\in \{1,\ldots, n\}| \varepsilon_i=\varepsilon_{n+1}\}$. Then  $$x=\varepsilon_{n+1} \sum_{i\in I} e_i\in B_{\Lip_0(K_{n+1})}.$$ 
To see that $B_{\Lip_0(K_{n+1})} \subset D$ we simply reverse our previous observation. So if we take $x\in C$ so $x=\varepsilon_{n+1} \sum_{i\in I} e_i$ where $I\subset\{1,\ldots, n\}$.  Then we define $\varepsilon_i=\varepsilon_{n+1}$ if $i\in I$ and $\varepsilon_i=-\varepsilon_{n+1}$ if $i\notin I$.  Then by our choices of $\varepsilon_i$ we have $x=\frac{1}{2} \sum_{i=1}^n  \varepsilon_i  e_i + \frac{1}{2} \varepsilon_{n+1} \sum_{i=1}^n e_i \in D$ as desired.
\end{proof}

\begin{remark} It follows from Claim \ref{cl_zono}  that $B_{\Lip_0(K_{n+1})}$ is a zonotope. This can be seen in a different way (see \cite{Go} and also \cite[Example 10.13]{Os}) by showing that $K_{n+1}$ embeds into a tree with ${n+2}$ points or by showing that $B_{\mathcal F(K_{n+1})}$ is a section of $\ell_1^{n+1}$. We also note that using the fact that the Mahler conjecture is true for zonotopes \cite{GMR} we get $\PP(K_{n+1}) \ge \PP(B_1^n)$. We will show it as a direct computation below.
\end{remark}

Let us compute the volume product of $K_{n+1}$.

\begin{claim} $$\PP(K_{n+1})=\frac{n+1}{n!} \binom{2n}{n}.$$
\end{claim}

\begin{proof} 
Let us first compute the volume of $B_{\Lip_0(K_{n+1})}$. Let $e=\sum_{i=1}^ne_i$. Then $B_{\Lip_0(K_{n+1})}$ is a zonotope which is the following sum of $n+1$ segments:
$$B_{\Lip_0(K_{n+1})}=\frac{1}{2} \left(\sum_{i=1}^n[-e_i,e_i]+[-e,e]\right).$$
Thus one may use the following formula for the volume of zonotopes \cite{Sh}: if $Z=\sum_{i=1}^m[0,u_i]$ is a zonotope which is the sum of $m$ vectors $u_1, \dots, u_m$ in $\RR^n$, with $m\ge n$ then 
$$
|Z|=\sum_{\card(I)=n}|\det(u_i)_{i\in I}|.
$$
Since for all $k\in\{1,\dots, n\}$ one has $|\det(e, (e_i)_{i\neq k})|=1$ we get $|B_{\Lip_0(K_{n+1})}|=n+1.$\\

Now let us compute the volume of $B_{\mathcal F(K_{n+1})}=\conv\{\pm e_i, \pm (e_i-e_j); 1\leq i\neq j\leq n\}$. We decompose $B_{\mathcal F(K_{n+1})}$ using the partition of $\RR^n$ into $2^n$ parts defined according to the coordinate signs. For $I\subset \{1,\dots n\}$ let 
$$
C_I=\{x\in\RR^n: x_i>0 \ \hbox{for all}\ i\in I\ \hbox{and}\ x_j<0 \ \hbox{for all}\ j\notin I\}.
$$
Then 
$$
B_{\mathcal F(K_{n+1})}\cap C_I=\conv(0; (e_i)_{i\in I}; (-e_j)_{j\notin I}; (e_i-e_j)_{i\in I, j\notin I})=\conv(0; (e_i)_{i\in I})-\conv(0; (e_j)_{j\notin I}).
$$
Hence if we denote $k=\card(I)$ then 
$$
|B_{\mathcal F(K_{n+1})}|=|\conv(0; (e_i)_{i\in I})|_k|\conv(0; (e_j)_{j\notin I})|_{n-k}=\frac{1}{\card(I)!\card(I^c)|!}=\frac{1}{k!(n-k)!}.
$$
Thus 
$$
|B_{\mathcal F(K_{n+1})}|=\sum_{I\subset\{1,\dots,n\}}\frac{1}{\card(I)!\card(I^c)|!}=\sum_{k=0}^n\binom{n}{k}\frac{1}{k!(n-k)!}=\frac{1}{n!}\sum_{k=0}^n\binom{n}{k}^2=\frac{1}{n!}\binom{2n}{n}.
$$
\end{proof}

\begin{remark} Note that a  more general case is the metric space for which $d_{i,j}\in \mathbb Z$ for all $i, j$. In this case the polytope $B_{\Lip_0(M)}$ has vertices in the lattice $\mathbb Z^n$. Indeed, a result of Farmer \cite{Fa} ensures that a Lipschitz function $f$ is an extreme point of $B_{\Lip_0(M)}$ if and only if for all $i, j$ there are $k_0, k_1,\ldots k_l$ such that $k_0=i$, $k_l=j$ and
\[ |f(a_{k_r})-f(a_{k_{r+1}})|  = d(a_{k_r}, a_{k_{r+1}}) \]
for all $r=0, \ldots, l-1$, and so $f(a_j)\in \mathbb Z$ for every $a_j\in M$. 

If, moreover, all the edges of the canonical graph associated with $M$ have weight $1$, then the vertices of $B_{\mathcal F(M)}$ are of the form $e_i-e_j$ for some $i,j$ and so $B_{\mathcal F(M)}$ is also a polytope with vertices in the lattice $\mathbb Z^n$. Thus it interesting to ask for minimal and maximal volume product among lattice polytopes whose dual is also a lattice polytope. 
\end{remark}

\subsection{The Minimal Case}\label{sec:minimal}

In this section, we focus on Conjecture \ref{conj:minimal}. Note that it is already known that this conjecture holds for certain metric spaces. Namely: 
\begin{itemize}
	\item If $M$ embeds into a tree, then $B_{\Lip_0(M)}$ is a zonoid \cite{Go} and so the result follows from \cite{R1,GMR}.
	\item If $M$ is a cycle, then $B_{\mathcal F(M)}$ has $2n+2$ vertices and so the result follows from the recent paper \cite{Ka}.
\end{itemize}	

We will show some more cases were Conjecture 2.3 holds. 

Given a connected graph $G=(V,E)$, we say that an edge $e\in E$ is a \emph{bridge} if the subgraph $(V, E\setminus\{e\})$ is disconnected. 

\begin{proposition}\label{prop:doublecone}
	Let $(M,d)$ be a finite metric space such that $\mathcal P(M)$ is minimal. Let $m_{i,j}$ be a vertex of $B_{\mathcal F(M)}$. Assume that all the facets of $B_{\mathcal F(M)}$ containing $m_{i,j}$ are simplices. Then, the edge $\{a_i,a_j\}$ is a bridge in $G(M,d)$. 
\end{proposition}
\begin{proof}
	Let $G=(V,E,w)$, be the graph of $M$. Let $G^t=(V,E,w^t)$ be the weighted graph obtained from $G$ by replacing the weights by $w^t(\{a_i,a_j\})=w(\{a_i,a_j\})/(1+t)$ and $w^t(\{u,v\})=w(\{u,v\})$ for any $\{u,v\}\in E\setminus\{\{a_i,a_j\}\}$. By Lemma~\ref{lemma:charmetricgraphs}, if $|t|$ is small enough then there is a metric $d^t$ on $M$ such that $G^t=G(M,d^t)$. Moreover, $t\mapsto |B_{\mathcal F(M,d^t)}|$ is affine if $|t|$ is small enough since all the faces containing $m_{i,j}$ are simplices. By the minimality of $\mathcal P(M)$ and Lemma~2 from \cite{FMZ}, $B_{\mathcal F(G)}$ must be a double-cone with apex $m_{i, j}$. Thus, $\operatorname{span}(\ext (B_{\mathcal F(M)})\setminus\{\pm m_{i,j}\})$ has dimension $n-1$. We claim that this implies that $\{a_i,a_j\}$ is a bridge in $G$. Indeed, otherwise there would exists a spanning tree in $G$ not containing the edge $\{a_i, a_j\}$, and the vector space generated by the corresponding molecules would be $n$-dimensional by Lemma~\ref{lemma:basis}.  
\end{proof}

Let us remark that it is also possible to give an alternative proof of Proposition~\ref{prop:doublecone} using the characterization of isometric Lipschitz-free spaces given in Theorem~\ref{th:isometries} and the properties of a shadow system. Indeed, by \cite{MR2} we have that $B_{\mathcal F(M,d^t)}$ is a linear image of $B_{\mathcal F(M)}$. Then there is a cyclic bijection between the edges of $G(M,d)$ and the ones of $G(M, d^t)$ such that $d_t/d$ is constant on each 2-connected component. Thus the edge $\{a_i,a_j\}$ is the only element of its 2-connected component, that is, it is a bridge. 

\begin{theorem} Let $(M,d)$ be a metric space with minimal volume product such that $B_{\mathcal F(M)}$ is simplicial. Then $M$ is a tree. 
\end{theorem}

\begin{proof}
	Apply Proposition~\ref{prop:doublecone} to get every edge in $G(M,d)$ is a bridge. %Thus, $G(d)$ is a tree. , so $B_{\mathcal F(M)}$ must be a linear image of $B_1^n$ so that $M$ must be a tree.
\end{proof}

We note that the minimal case for four points corresponds to the question on the minimality of volume product in $\RR^3$. That question was recently solved in \cite{IS} and automatically gives $\PP(M)\geq \PP(B_1^3),$ for $M$ being a metric space of four elements. Here we present a direct proof that uses the structure of polytope of $B_{\mathcal F (M)}$. To that end, we need the following simple observation. 

\begin{lemma}\label{lemma:edgedisappears}  Assume that there is a face of $B_{\mathcal F(M)}$ containing the molecules $m_{i,j}$ and $m_{j,k}$. Then $m_{i,k}$ is not a vertex of $B_{\mathcal F(M)}$.
\end{lemma}
\begin{proof}
	In that case there is a $1$-Lipschitz function $f$ such that  $f(a_i)-f(a_j)=d_{i,j}$ and $f(a_j)-f(a_k)=d_{j,k}$. Thus
	\[ d_{i,j}+d_{j,k}=f(a_i)-f(a_k)\leq d_{i,k}\]
	and so $m_{i,k}$ is not a vertex. 
\end{proof}

\begin{theorem} Let $(M,d)$ be a metric space with four points. Then $\mathcal P(M)\geq \mathcal P(B_1^3)$. Moreover, the equality holds if and only if $M$ is a tree or a cycle with equal weights. 
\end{theorem}
\begin{proof}
	Let $G=(V,E,w)$ be the canonical graph associated with $M$. Note first that if the graph $G$ has a leaf, then $M=M_1\diamond M_2$, where $\#M_1=1$ and $\#M_2=3$. Then $\mathcal P(M)\geq \mathcal P(B_1^3)$ follows from \eqref{eq:diamond} and the two dimensional case of Mahler's conjecture. Moreover, if equality holds then $M_2$ has minimal volume product and so it is a tree, thus $M$ is a tree too. 
	
	Thus, we may assume that $G$ does not any leaf. Assume also that $M$ has minimal volume product among all metric spaces with four elements. Then every vertex of $B_{\mathcal F(M)}$ belonging to a facet that contains at least four vertices. Indeed, otherwise Proposition~\ref{prop:doublecone} says that $G$ has a bridge, and so it has a leaf. Therefore, every edge in $G$ belongs to a cycle, which has length $4$ by Lemma~\ref{lemma:edgedisappears}. So either $G$ is a cycle or it is a complete graph. In the first case $B_{\mathcal F(M)}$, has $8$ vertices that lie in two parallel facets. Then \cite{FMZ} (see also \cite{LopezReisner, Ka}) says that $\mathcal P(M)\geq \mathcal P(B_1^3)$, and if equality holds then either $B_{\mathcal F(M)}$ is a double cone or it is affinely isometric to $B_\infty^3$. In the first case, $6$ of the vertices of $B_{\mathcal F(M)}$ would be coplanar and it is easy to check that is not possible. On the other hand, one can check that if $\mathcal F(M)$ is isometric to $\ell_\infty^3$ then $M$ is a cycle with equal weights, this also follows from Corollary~\ref{cor:ellinfty}. 
	
	Thus, it remains to check the case in which $G$ is a complete graph. By relabeling the points of $M$ we may assume that
	\[d_{1,0}+d_{2,3}\leq d_{2,0}+d_{1,3}\leq d_{3,0}+d_{1,2}.\]
	Let $F$ be a facet of $B_{\mathcal F(M)}$ containing $m_{3,0}$ and three other vertices. Thanks to Lemma~\ref{lemma:edgedisappears}, there are only two possible cases:

	\emph{Case 1. $F$ contains $m_{3,0}$, $m_{3,2}$, $m_{1,2}$ and $m_{1,0}$.} Consider the following determinant.\\
	$$\Delta=\left|\begin{array}{cccc}
	1 			& 1 			& 1 & 1       \\
	0 			& 0 			& d_{1,2}^{-1}   & d_{1,0}^{-1}     \\
	0 			& -d_{2,3}^{-1} & -d_{1,2}^{-1}  & 0     \\
	d_{3,0}^{-1} & d_{2,3}^{-1} & 0				 & 0
	\end{array}\right|=\frac{d_{3,0}+d_{1,2}-d_{2,3}-d_{1,0}}{d_{1,0}d_{1,2}d_{3,0}d_{2,3}}$$
	We have that $\Delta=0$ since the four molecules $m_{3,0}$, $m_{3,2}$, $m_{1,2}$ and $m_{1,0}$ are coplanar, and so $d_{1,0}+d_{2,3}=d_{3,0}+d_{2,3}$.
	
	\emph{Case 2. $F$ contains $m_{3,0}$, $m_{3,1}$, $m_{2,1}$ and $m_{2,0}$}. The same argument as before yields $d_{2,0}+d_{1,3}=d_{3,0}+d_{1,3}$.
	
	We conclude that in any case $d_{2,0}+d_{1,3}=d_{3,0}+d_{1,3}$. Then $M$ satisfies the four-point condition and so $B_{\Lip_0(M)}$ is a zonoid \cite{Go}. Therefore $\mathcal P(M)\geq \mathcal P(B_1^3)$. Moreover, equality does not hold in this case since that would imply that $B_{\mathcal F(M)}$ is affinely isomorphic to $B_\infty^3$ \cite{GMR}, so it would have $8$ vertices. However, $B_{\mathcal F(M)}$ has $12$ vertices since $G$ is complete. 
\end{proof}	 

\iffalse
The following examples gather the main difficulties that we find when we try to extend the result to higher dimensions.

\begin{example}
	Consider the metric space $M=\{0,1,2,3,4\}$ whose associated graph contains the edges $\{0,1\}$,$\{1,2\}$,$\{2,3\}$,$\{3,0\}$, $\{2,4\}$ and $\{3,4\}$ with weight 1. Then:
	\begin{itemize}
		\item Each vertex of $B_{\mathcal F(M)}$ belongs to a facet with $5$ vertices (thus, the facet is not simplicial).
		\item $B_{\Lip_0(M)}$ is not a zonoid because $M$ does not satisfy the four-points condition. 
		\item $B_{\mathcal F(M)}$ has $12>10$ vertices. 
	\end{itemize} 
\end{example}
\fi

\textbf{Acknowledgements}: The authors thank Matthieu Meyer and Gideon Schechtman for their valuable suggestions.

\end{document}